\documentclass[10pt,oneside,a4paper,reqno]{amsart}  
\usepackage{mathrsfs}
\usepackage{amssymb}

\usepackage{amssymb, bbm,
enumerate}
\usepackage{geometry}                
\usepackage[T1]{fontenc}

\usepackage{mathrsfs} 

\usepackage[english]{babel}
\usepackage{graphicx}
\usepackage{color}
\usepackage[colorlinks=true, linkcolor=magenta, citecolor=cyan, urlcolor=blue]{hyperref}

\usepackage{scalerel,stackengine}
\stackMath
\newcommand\reallywidehat[1]{%
\savestack{\tmpbox}{\stretchto{%
  \scaleto{%
    \scalerel*[\widthof{\ensuremath{#1}}]{\kern-.6pt\bigwedge\kern-.6pt}%
    {\rule[-\textheight/2]{1ex}{\textheight}}
  }{\textheight}%
}{0.5ex}}%
\stackon[1pt]{#1}{\tmpbox}%
}
\renewcommand{\phi}{\varphi}

\newcommand{\meanv}[1]{\langle#1\rangle}

\newcommand{\mc}[1]{\mathcal{#1}}

\def\be{\begin{equation}}
\def\ee{\end{equation}}
\def\bea{\begin{eqnarray}}
\def\eea{\end{eqnarray}}
\def\ni{\noindent}
\def\nn{\nonumber}
\def\T{\mathbb{T}}

\def\Z{\mathbb{Z}}
\def\N{\mathbb{N}}

\DeclareMathSymbol{\leqslant}{\mathalpha}{AMSa}{"36} 
\DeclareMathSymbol{\geqslant}{\mathalpha}{AMSa}{"3E} 
\DeclareMathSymbol{\eset}{\mathalpha}{AMSb}{"3F}     
\renewcommand{\leq}{\;\leqslant\;}                   
\renewcommand{\geq}{\;\geqslant\;}                   

\DeclareMathOperator{\Id}{Id}

\DeclareMathOperator{\Span}{span}

\def\ie{\textit{i.e. }}

\def\d{\delta}
\def\g{\gamma}

\def\b{\beta}
\def\D{\Delta}
\def\l{\lambda}

\def\s{\sigma}

\def\k{\kappa}

\def\R{\mathbb{R}}

\theoremstyle{plain}
\newtheorem{theorem}{Theorem}[section]
\newtheorem{lemma}[theorem]{Lemma}
\newtheorem{proposition}[theorem]{Proposition}

\theoremstyle{definition}

\theoremstyle{remark}
\newtheorem{remark}[theorem]{Remark}

\numberwithin{equation}{section}

\definecolor{light}{gray}{.9}

\author{Giuseppe Genovese}
\address{Department Mathematik und Informatik, Universit\"at Basel Spiegelgasse 1, CH-4051 Basel, Switzerland}
\email{giuseppe.genovese@unibas.ch}

\author{Renato Luc\`a}
\address{BCAM - Basque Center for Applied Mathematics, 48009 Bilbao, Spain and Ikerbasque, Basque Foundation
for Science, 48011 Bilbao, Spain.}
\email{rluca@bcamath.org} 

\author{Nikolay Tzvetkov}
\address{Department of Mathematics (AGM), University of Cergy-Pontoise, 2, av. Adolphe Chauvin, 95302 Cergy-Pontoise Cedex, FRANCE}
\email{nikolay.tzvetkov@u-cergy.fr}

\title
{Quasi-invariance of Gaussian measures for the periodic Benjamin-Ono-BBM equation}

\date{\today}

\makeindex
\begin{document}
\begin{abstract}
The BBM equation is a Hamiltonian PDE which revealed to be a very interesting test-model to study the transformation property of Gaussian measures along the flow, after~\cite{sigma}. In this paper we study the BBM equation with critical dispersion (which is a Benjamin-Ono type model). We prove that the image of the Gaussian measures supported on fractional Sobolev spaces of increasing regularity are absolutely continuous, but we cannot identify the density, for which new ideas are needed. 
\end{abstract}

\maketitle

\section{Introduction}

We study the equation of the Benjamin-Ono type
\begin{equation}\label{BBM-gamma}
\partial_t u +\partial_t|D_x| u +\partial_x u + \partial_x (u^2)=0, \quad u(0,x)=u_0(x)\,,
\end{equation}
posed on the one dimensional flat torus $\T := \R/2\pi \Z$, where 
$$
|D_x| u (x):=\sum_{n \neq 0}|n|\hat{u}(n)e^{inx}  \,.
$$
In \eqref{BBM-gamma} $u$ is real valued and we shall consider only zero average solutions 
(note that the average is preserved  
under the evolution). The global existence and uniqueness of solutions for \eqref{BBM-gamma} has been obtained in \cite{mammeri} for data in $H^s$ with $s>\frac12$.
We denote by $\Phi_t$ the associated flow and write $u(t) := \Phi_tu$. 
It can be proved that the quantity
\be\label{eq:conservata}
\mathcal E[u]:= \left( \|u\|^2_{L^2} + 4 \pi \|u\|^2_{H^\frac12} \right)^{1/2}\,
\ee
is conserved by the evolution of \eqref{BBM-gamma}; see Lemma 2.4 in \cite{sigma}.
In particular, the norm $H^{\frac12}$ remains bounded in time. 

We are interested in the transformation property along $\Phi_t$ of the Gaussian measures defined as follows. Let $\{h_n\}_{n\in\N}$, $\{l_n\}_{n\in\N}$ be two independent sequences of independent $\mathcal N(0,1)$ random variables. 

Set
$$
g_n:=\begin{cases}
\frac{1}{\sqrt{2}}(h_n+il_n)&n\in\N\\
\frac{1}{\sqrt{2}}(h_n-il_n)&-n\in\N\\
\end{cases}
\,.
$$
We denote by $\g_{s}$ the Gaussian measure on~$H^{s-}$ induced by the random Fourier series 
\begin{equation}\label{Def:GaussmEasure}
\varphi_s(x)=\sum_{n\neq 0}\frac{g_n}{|n|^{s+1/2}}e^{inx}\,
\end{equation}
and $E_s$ the associated expectation value.
We also define
\be\label{Intro:eq:ref-meas}
\tilde\g_{s}(A) :=      E_s[1_{A \cap\{\mc E[u]\leq R\}}] \,.
\ee

Our main result is the quasi-invariance of $\tilde\g_s$ along $\Phi_t$. 
We extend previous achievements of \cite{sigma, BBM1} for $\b>1$ (in the notation of \cite{BBM1}) to the minimal dispersion case $\b=1$. 

We recall that a measure $\mu$ on a space $X$ is quasi-invariant with respect to a map $\Phi : X \to X$ if
its image under $\Phi$ is absolutely continuous with respect to $\mu$. 

\begin{theorem}\label{TH:quasi}
Let $s > 1$ and $R >0$.
Then the measures $\tilde	\g_{s}$ are quasi-invariant along the flow of \eqref{BBM-gamma}. For any $t\in\R$ there is 
$p= p(t, R)>1$ such that the densities of the transported measures lie in $L^{p}(\tilde\g_{s})$. 
\end{theorem}

\begin{remark}\label{Remark:Rtoinfty}
The cut-off $1_{\{\mc E[u]\leq R\}}$ used to define the measure 
$\tilde \g_{s}$ allows to prove not only the quasi-invariance of the transported measure, 
but the fact that it belongs to $L^{p}(\tilde\g_{s})$ for some $p  >1$, at any time. 
On the other hand, if we are only interested to the quasi-invariance of the reference measure 
we can upgrade the result to the limit case $R \to \infty$, namely proving the
quasi-invariance 
of the transport of the measure $\g_s$ under $\Phi_t$, for all $t \in \R$. For details we refer to Section 3.2 of \cite{FS}. 
\end{remark}

The transport of Gaussian measures under given transformations is a classical subject of probability theory, starting from the classical works of Cameron-Martin \cite{CM} for shifts and Girsanov \cite{girs} for non-anticipative maps (\ie adapted). The anticipative (or non-adapted) case is more difficult to deal with and a crucial role is played by the generator of the transformation. Kuo \cite{kuo2} established a Jacobi formula which generalises the Girsanov formula in case of maps with trace-class generator and Ramer \cite{Ramer} extended it to Hilbert-Schmidt generators (for a comparison of the Girsanov and Ramer change of variable formula see \cite{ofer}). Further developments have been achieved in the context of Malliavin calculus, see for instance \cite{ABC1,ABC2, ust}, essentially establishing Jacobi formulas for Gaussian measures in functional spaces for more general classes of maps.

In \cite{sigma} the third author introduced a new method to prove quasi-invariance of Gaussian measures along the flow of dispersive PDEs. This paper triggered a renewed interest in the subject from the viewpoint of dispersive PDEs, which translates into studying the evolution of random initial data (such as Brownian motion and related processes). For recent developments on the topic, see \cite{BT, deb, FS,forl, gauge, BBM1, GLV0, GLV, GOTW,OS,OT1,OT2,OT3,OTT,PTV,phil}, although this list might be not exhaustive. 

The technique of \cite{sigma} permits to treat flow maps whose differential is not in the Hilbert-Schmidt class, thus improving on the classical results. However it is only used to prove absolute continuity of the transported Gaussian measure without providing an explicit approximation of the density of the infinite dimensional change of coordinates induced by the flow, which is an important open question for many Hamiltonian PDEs and related models. 
Recent progresses in this direction have been made in \cite{NR-BSS11}, \cite{deb}, \cite{BBM1},
where the techniques developed allowed to get an exact formula for the density.  

Identifying the density is one major difference between the present work and \cite{BBM1}. Indeed whereas in \cite{BBM1} we could prove the strong convergence of a sequence of approximating densities, here the minimal amount of dispersion in (\ref{BBM-gamma}) prevents us to employ the same method and the question remains open. This challenging problem presents similarly also for the DNSL gauge group (see discussion below). 
Developing a robust technique to identify the density of the 
transported measures in these kind of problems would represent the completion of the programme started in \cite{sigma}. Another important difference is that here we work without the exponential cut-off on the $H^s$ norm, which was central in our previous work \cite{BBM1}, and this introduces a number of technical difficulties. In particular we have to use finer probabilistic estimates. 

The proof relies on the study of the derivative at time zero of the $H^{s+\frac12}$ Sobolev norm
\be\label{eq:intro-F}
F[u]=\frac{d}{dt}\|\Phi_t u\|^2_{H^{s+\frac12}}\Big|_{t=0}\,
\ee
which plays the key role in the method introduced by the third author in \cite{sigma}. The crucial point of the present work is that this object is not trivial to bound in any sense w.r.t. $\tilde\g_s$. Indeed we prove that it behaves as a sub-exponential random variable, which gives precisely the endpoint estimate on the $L^p(\tilde\g_s)$ norm in order to apply the argument of \cite{sigma}. We faced the same kind of difficulties studying the gauge group associated to the derivative nonlinear Schr\"odinger equation, and indeed our strategy unrolls similarly as in \cite{gauge}. The analogies between the DNSL gauge and the equation (\ref{BBM-gamma}) are interesting. Despite their simplicity, both models are critical for the method of \cite{sigma} in the sense mentioned above: the term \eqref{eq:intro-F} is a sub-exponential random variable w.r.t. the Gaussian measure and, more importantly, it cannot be bounded within its support in terms of Sobolev norms (another notable example is the nonlinear wave equation studied in \cite{OT2, GOTW}, where the renormalisation needed in higher dimension complicates things further). Intuitively, this is due to the minimal amount of dispersion in the models, which for equation (\ref{BBM-gamma}) can be clearly seen by comparison with the usual BBM equation with dispersion parameter strictly greater than one (denoted by $\g$ in \cite{sigma}, $\b$ in \cite{BBM1}).

We conclude discussing the assumption $s >1$. This assumption is only used in Lemma 
\eqref{lemma:subexp-inf}, where is necessary in order for $\|\partial_xu\|_{L^{\infty}}$ to be finite 
$\g_s$-almost surely. 
On the other hand, the other
probabilistic arguments that we used only requires $s >1/2$. It would be certainly interesting trying to relax $s>1$, 
however this
would require new ideas in order to avoid the use of Lemma~\eqref{lemma:subexp-inf}.

The rest of the paper is organised as follows. Section \ref{sect:smoothing} contains the necessary deterministic estimates for of the $H^{s+\frac12}$ norm at $t=0$ ($F$ defined above). In Section \ref{sect:Wick} we prove the convergence in $L^2(\tilde\g_s)$ of suitable truncations of $F$. In Section \ref{sect:prob} we show that $F$, as a random variable w.r.t. $\tilde\g_s$, has an exponential tail. Finally in Section \ref{section:quasi1} we ultimate the proof of the main result.

{\bf Notations.} 
Given a function $f : \T \to \R$ with zero average, we define its Sobolev norm $H^{s}$ as 
$$
\| f \|_{H^s}^2 := 
\sum_{n=1}^{\infty}|n|^{2 s}|\hat{f}(n)|^2\,.
$$
Note that with this definition the norms $L^2$ and $H^0$ differs by a factor $2 \sqrt{\pi}$. 
This is why this factor appears in \eqref{eq:conservata}. 
A ball of radius $R$ and centred in zero in the $H^s$ topology is denoted by $B^s(R)$. We drop the superscript for $s=0$ (ball of $L^2$). We write $\meanv{\cdot} := (1 + |\cdot|^2)^{1/2}$.
We write $X\lesssim Y$ if there is a constant $c>0$ such that $X\leq cY$ and $X\simeq Y$ if $Y\lesssim X\lesssim Y$. We underscore the dependency of $c$ on the additional parameter $a$ writing $X\lesssim_a Y$. $C,c$ always denote constants that often vary from line to line within a calculation. 
We denote by $P_N$ the orthogonal projection on Fourier modes $\leq N$, namely 
$$
P_N(u)=\sum_{|n|\leq N}\hat{u}(n)e^{inx}\,,
$$
where $\hat{u}(n)$ is the $n$-th Fourier coefficient of $u\in L^2$. 
Also, we denote the Littlewood-Paley projector by $\D_0:=P_{1}$, 
$\D_j:=P_{2^{j}}-P_{2^{j-1}}$, $j \in \N$. We use the standard notation $[A,B] := AB - BA$ to denote the commutator of the operators $A, B$. We will use the following well-known tail bounds for sequences of independent centred Gaussian random variables $X_1,\ldots,X_d$ (see for instance \cite{ver}): 
\be\label{eq_Hoeffing}
P\left(\left| \sum_{i=1}^d |X_i| - E[\sum_{i=1}^d |X_i|] \right| \geq \lambda \right)\leq C\exp\left(-c\frac{\lambda^2}{d}\right)\,
\ee
and the Bernstein inequality
\be\label{eq_Bernstein}
P\left( \left|\sum_{i=1}^d |X_i|^2-E[\sum_{i=1}^d |X_i|^2]\right| \geq \lambda \right)\leq C\exp\left(-c\min\left(\l,\frac{\lambda^2}{d}\,\right)\right)\,.
\ee


\subsection*{Acknowledgements}
R. Luc\`a is supported by the Basque Government under program BCAM- BERC 2018-2021, 
by the Spanish Ministry of Science, Innovation and Universities under the BCAM Severo Ochoa 
accreditation SEV-2017-0718 and by IHAIP project PGC2018-094528-B-I00 (AEI/FEDER, UE). 
N. Tzvetkov is supported by ANR grant ODA (ANR-18-CE40-0020-01). The authors would like to thank the anonymous referees
for many useful comments that helped to improve the presentation (see for instance Remark \ref{Remark:Rtoinfty}).

\section{Smoothing estimates}\label{sect:smoothing}

We will work with the truncated equation
\begin{equation}\label{BBM-gamma-NGamma=1}
\partial_t u+\partial_t|D_x| u+\partial_x u+\partial_xP_N((P_Nu)^2)=0, \quad u(0,x)=u_0(x)\,.
\end{equation}
We denote by $\Phi^N_t$ the associated flow and $\Phi^{N=\infty}_t=\Phi_t$.
The flow $\Phi^N_t$ is well defined, since the global well-posedness 
of \eqref{BBM-gamma-NGamma=1} is clear by the fact that at fixed $N$  
the nonlinear part of the evolution regards only the Fourier modes in $[-N,N]$. One has a local existence time which 
depends on the $L^{2}$ norm on the initial datum (this clearly fails for $N = \infty$), and then one can globalize the
solutions obtained using the Cauchy theorem taking advantage of the invariance of~\eqref{eq:conservata}
under the flow~$\Phi^N_t$; see Lemma~2.4 in \cite{sigma}.

The crucial quantity we deal with is
\be\label{eq:FN}
F_N(t,u) := \frac{d}{dt}
\| P_N \Phi_t^N u\|^2_{H^{s+\frac12}}\,. 
\ee
We will abbreviate $F_N = F_N(0,u)$.

In this section all the integrals are taken over $x \in \T$. We always  omit the $dx$ to simplify the notations.
\begin{proposition}\label{prop:growth-Hs-norm}
We have 
\begin{equation}\label{DI1}
F_N(t,u) =
F_{1,N}(t,u) + F_{2,N}(t,u) + F_{3,N}(t,u),
\end{equation}
where
\begin{equation}\label{Def:F1N}
F_{1,N}(t,u) := \frac{2}{\pi} \int_{\T} \big| |D|^{s}P_N u(t) \big|^2 \partial_x P_Nu(t) ,
\end{equation}
\begin{equation}\label{Def:F2N}
F_{2,N}(t,u) :=  - \frac{4}{\pi} \int_{\T} (|D_x|^{s}P_N u(t))\,
 \big[ |D_x	|^{s}, P_N u(t)  \big] \partial_x P_N u(t),
\end{equation}
\begin{equation}\label{Def:F3N}
F_{3,N}(t,u) := \frac{2}{\pi} \int_{\T} (|D_x|^{s}P_N u(t))\,
  \frac{\partial_x |D_x|^{s}}{1+|D_x|}
 \big((P_Nu(t))^2 \big)\,.
\end{equation}
\end{proposition}

\begin{proof}
From \eqref{BBM-gamma-NGamma=1} we have 
\begin{equation}\label{En1Preq}
\partial_t   |D_x|^{\sigma}  P_N u(t) =  \frac{|D_x|^{\sigma}}{1 + |D_x| } \left(  -\partial_x P_N u(t) - \partial_x P_N ((P_Nu(t))^2) \right).
\end{equation}
Since
\begin{align}
\int_{\T}  & \left( \frac{|D_x|^{\sigma}}{1+ |D_x| } \left(  \partial_x P_N u(t)    \right)  \right) |D_x|^{\sigma} P_N u(t) 
\\ \nonumber
&=
\int_{\T}   \left(  \partial_x \left( \frac{|D_x|^{\sigma}}{\sqrt{1+ |D_x|} }    P_N u(t)    \right)  \right) \frac{|D_x|^{\sigma}}{\sqrt{1+ |D_x|} } P_N u(t) 
=
\frac{1}{2}  \int_{\T} \partial_x \left( \left( \frac{ |D_x|^{\sigma}}{\sqrt{1 + |D_x|}}  P_N u(t) \right)^2  \right) =0,
\end{align}
pairing \eqref{En1Preq} in $L^{2}$ with $|D_x|^{\sigma} P_N u(t)$ we can compute
\begin{align}\label{PivotB}
\frac{d}{dt}\|P_N u(t) \|_{H^\sigma}^2 
& = - \frac{2}{\pi} \int_{\T} \Big(|D_x|^{\sigma} P_N u(t) \Big)\, \frac{ |D_x|^{\sigma} }{1+|D_x|}\partial_x \left( (P_Nu(t) )^2 \right);
\end{align}
note that on the r.h.s. we can write $\partial_x ( (P_Nu )^2)$ in place of $P_N \partial_x ( (P_Nu )^2)$ by orthogonality.

Choosing $\sigma = s + 1/2$ into  \eqref{PivotB} we get
\begin{align}\nonumber
\frac{d}{dt}\|P_N u(t) \|_{H^{s + 1/2}}^2 
& = - \frac{2}{\pi} \int_{\T} \Big(|D_x|^{s + 1/2} P_N u(t) \Big) \,  \frac{ |D_x|^{s + 1/2}}{1+|D_x| } 
 \partial_x ( (P_N u(t) )^2) .
\end{align}
This implies 
\begin{align}\nonumber
\frac{d}{dt}\|P_N u(t) \|_{H^{s + 1/2}}^2 
& = - \frac{2}{\pi} \int_{\T} \Big( \frac{|D_x|}{1+|D_x|} |D_x|^s P_N u(t) \Big)\, |D_x|^{s}  \partial_x \big( (P_N u(t)  )^2 \big) .
\end{align}
Thus, writing
$$
 \frac{|D_x|}{1+|D_x|} = 1 - \frac{1}{1+|D_x|}
$$
we arrive to
$$
\frac{d}{dt}\|P_N u(t)\|_{H^{s+1/2}}^2=I_1(t)+I_2(t),
$$
where 
\bea
I_1 (t) &=& - \frac{2}{\pi} \int_{\T} (|D_x|^{s}P_N u(t))\,|D_x|^{s}
  \partial_x \big( (P_Nu(t))^2 \big)\nn\\
 &=&  - \frac{4}{\pi} \int_{\T} (|D_x|^{s}P_N u(t))\,|D_x|^{s}
 \big((\partial_x(P_Nu(t))) P_Nu(t) \big)\nn
\eea
and
\bea
I_2 (t)
 &=& \frac{2}{\pi} \int_{\T} \left( \frac{|D_x|^{s}}{1+|D_x|}P_N u(t) \right)\,
 |D_x|^{s}
 \big(\partial_x(P_Nu(t))^2 \big)  \nn\\
& =& 
\frac{2}{\pi} \int_{\T} (|D_x|^{s}P_N u(t))\,
  \frac{\partial_x |D_x|^{s}}{1+|D_x|}
 \big((P_Nu(t))^2 \big)\,.\nn
\eea

Using
$$
|D_x|^s \big( (\partial_x P_N u(t)) (P_N u(t)) \big) = \big( |D_x|^s \partial_x P_N u(t) \big) P_N u(t) + 
\big[ |D_x|^s, P_N u(t) \big] \partial_x P_N u(t)
$$
we can rewrite $I_1$ as
\begin{align}\label{I_1Decomp}
I_1 (t)
& =  - \frac{4}{\pi} \int_{\T} (|D_x|^{s}P_N u(t))\,( |D_x|^{s}
\partial_x P_N u(t)) (P_N u(t)) 
\\ \nonumber
& - \frac{4}{\pi} \int_{\T} (|D_x|^{s}P_N u(t))\,
 \big[ |D_x|^{s}, P_N u(t)  \big] \partial_x P_N u(t).
\end{align}
Integrating by parts we can rewrite the first term on the r.h.s. of \eqref{I_1Decomp} as
\begin{equation}\label{Def:TildeIBetter}
  \frac{2}{\pi} \int_{\T} \big| |D_x|^{s}P_N u(t) \big|^2 \partial_x P_Nu(t) .
\end{equation}
Thus we arrive to
\begin{equation}
I_1 (t)
=    \frac{2}{\pi} \int_{\T} \big| |D_x|^{s}P_N u(t) \big|^2 \partial_x P_Nu(t) 
 - \frac{4}{\pi} \int_{\T} (|D_x|^{s}P_N u(t))\,
 \big[ |D_x|^{s}, P_N u(t)  \big] \partial_x P_N u(t).
\end{equation}

Since $I_1(t) = F_{1,N}(t,u) + F_{2,N}(t,u)$ and $I_2(t) =F_{3,N}(t,u)$ the proof is concluded.
\end{proof}

\begin{proposition}\label{prop:energy}
Let $s \geq 0$. 
The solutions of \eqref{BBM-gamma-NGamma=1} satisfy for all~$t \in \R$: 
\begin{equation}\label{GammaSmoothing}
\left| \frac{d}{dt}\| P_Nu(t) \|_{H^{s+1/2}}^2 \right| \lesssim    \|P_N u(t)\|_{H^s}^2 \|\partial_xP_N u(t)\|_{L^{\infty}} \,.
\end{equation}
\end{proposition}
\begin{proof}
By \eqref{DI1} it suffices to show that
\begin{equation}\label{3Bound}
|F_{1,N}(t,u)| + |F_{2,N}(t,u)| + |F_{3,N}(t,u)| \lesssim  \|P_N u(t)\|_{H^s}^2 \|\partial_xP_N u(u)\|_{L^{\infty}} \, .
\end{equation}
This is immediate in the case of $F_{1,N}(t,u)$, by H\"older's inequality. 

 For $F_{2,N}(t,u)$ we use the following commutator estimate from
 \cite{KatoPonce}, valid for $f$ periodic with zero average:
$$
\big\| \big[  |D_x|^s, f \big] g  \big\|_{L^2} \lesssim \| \partial_x f \|_{L^{\infty}} \| g \|_{H^{s-1}} + \| f \|_{H^s} \| g \|_{L^{\infty}} \,.
$$
Since $P_N u(t)$ has zero average we obtain
\begin{align*}\nonumber
\big\| \big[ |D_x|^{s}, P_N u(t) \big] \partial_x P_N u(t)  \big\|_{L^{2}}
& \lesssim
\| \partial_x P_N u(t) \|_{L^{\infty}} \| \partial_x P_N u(t) \|_{H^{s-1}} 
\\ \nonumber&
+ \|  P_N u(t) \|_{H^s} \| \partial_x P_N u(t) \|_{L^{\infty}}
\lesssim \| P_N u(t) \|_{H^{s}} \| \partial_x P_N u(t) \|_{L^{\infty}} , 
\end{align*}
whence
$$
|F_{2,N}(t,u)| \simeq  \left| \int_{\T} (|D_x|^{s}P_N u(t))\,
 \big[ |D_x|^{s}, P_N u(t)  \big] \partial_x P_N u(t) \right|
 \lesssim 
 \| P_N u(t) \|_{H^{s}}^2 \| \partial_x P_N u(t) \|_{L^{\infty}}  \, .
$$

The contribution of $F_{3,N}(t,u)$ is even smaller. Indeed, 
since $\frac{\partial_x }{1+|D_x|}$ is bounded on $L^2$ we have
\begin{align}
|F_{3,N}(t,u)| & 
\simeq  \left| \int_{\T} (|D_x|^{s}P_N u)\,
  \frac{\partial_x |D_x|^{s}}{1+|D_x|} 
 \big((P_Nu)^2 \big) \right| 
 \\ \nonumber
 &
 \lesssim \|P_N u(t)\|_{H^s}\|(P_N u(t))^2\|_{H^s} \lesssim \|P_N u(t)\|_{H^s}^2 \|  P_N u(t)\|_{L^{\infty}}
\end{align}
where we used $\| f g \|_{H^{s}} \lesssim \| f  \|_{H^{s}} \| g \|_{L^{\infty}} + \| g  \|_{H^{s}} \|f \|_{L^{\infty}}$
with $f=g=P_N u(t)$ in the last bound. Then~\eqref{3Bound} for $F_{3,N}(t,u)$ follows 
since $P_N u(u)$ has zero average, thus 
$\|  P_N u(t)\|_{L^{\infty}} \lesssim \|  \partial_x P_N u(t)\|_{L^{\infty}}$.

\end{proof}


\section{Second moment estimates}\label{sect:Wick}

The goal of this section is to prove the $L^2(\g_s)$ convergence of the term $F_N$ (recall once again that we are using the simplified notation $F_N=F_N(0,u)$, namely 
we mean that the time derivative in (\ref{eq:FN}) is evaluated at~$t=0$). This is a crucial result in our paper, as it allows us to exploit the random cancellations in $F_N$ to get bounds on this quantity which appear prohibitive to achieve deterministically.

\begin{proposition}\label{prop:Wick}
For all $N>M$ it holds
\bea
\| F_{N} - F_{M} \|_{L^{2}(\gamma_s)}& \lesssim&\frac{1}{M^{\frac{2s-1}{4}}}\,,\quad s\in(\frac12, \frac32]\label{DecayF-s-piccolo}\\
\| F_{N} - F_{M} \|_{L^{2}(\gamma_s)} &\lesssim& \frac{1}{\sqrt{M}}\,, \quad s>\frac32\label{DecayF-s-grande}\,.
\eea
\end{proposition}

The above result concerns the full range $s>\frac12$, even though in the rest of the paper only the case $s>1$ will be studied.
Considering $s>\frac12$ here could be however useful in a future attempt to 
extend also the main results of this paper to $s > \frac12$.

We start by a simple result on the decay of discrete convolutions. The way we use these bounds is explained in Remark \ref{rmk:conv1}.

\begin{lemma}\label{lemma:conv1}
Let $M\in\N$. 
Let $x,y>0$ with $x + y>1$. Let $p,q\geq1$ such that $\max(\frac1p,\frac1q)<x$ and $\max(\frac{p-1}{p},\frac{q-1}{q})<y$. Then there is $c=c(p,q,x,y)>0$ such that
\bea
\sum_{n\in\Z}\frac{1}{\meanv{n}^x\meanv{m-n}^y}&\leq& \frac{c}{\meanv{m}^r}\,,\quad r:=\min\left(x-\frac1p,\frac1q-(1-y)\right)\label{eq:conv0}\\
\sum_{|n|\geq M}\frac{1}{\meanv{n}^x\meanv{m-n}^y}&\leq& \frac{c1_{\{|m|\geq\frac{2M}{3}\}}}{\meanv{m}^{x-\frac1p}}+\frac{c}{\meanv{M}^{x-\frac1q}\meanv{m}^{\frac1q-(1-y)}}\,.\label{eq:conv1}
\eea
\end{lemma}
\begin{remark}\label{rmk:conv0}
If $x,y>1$, 
taking $p=\infty$ and $q=1$ we recover the well-known convolution estimate for powers. 
\end{remark}

\begin{proof}
We can assume $m \neq 0$, otherwise the statement is immediate. 
We have
\bea\label{G-Est}
\sum_{|n| \geq M}\frac{1}{\meanv{n}^x\meanv{m-n}^y}&\leq & 1_{\{|m|\geq \frac{2M}3\}}\sum_{\{|n|\geq \frac{|m|}{2}\}}\frac{1}{\meanv{n}^x\meanv{m-n}^y}\nn\\
&+&\sum_{\{|m-n|\geq \frac{|m|}{2}\}\cap\{|n|\geq M\}}\frac{1}{\meanv{n}^x\meanv{m-n}^y}\,.
\eea
Note that the  second summand is sufficient to bound the l.h.s. if $M>\frac{3|m|}{2}$, otherwise the first term is needed.

We estimate separately the two summands by the H\"older inequality. 
We have
\be
\sum_{\{|n|\geq\frac{|m|}{2}\}}\frac{1}{\meanv{n}^x\meanv{m-n}^y}
\leq \left(\sum_{\{|n|\geq\frac{|m|}{2}\}}\frac{1}{\meanv{n}^{xp}}\right)^{\frac1p}
\left(\sum_{ n\in\Z }\frac{1}{\meanv{n}^{\frac{yp}{p-1}}}\right)^{\frac{p-1}{p}}\leq
\frac{c_1(p,y)}{|m|^{x-\frac1p}}\,\nn
\ee
Similarly 
\bea
\sum_{\{|n|\geq M\}\cap\{|m-n|\geq \frac{|m|}{2}\}}\frac{1}{\meanv{n}^x\meanv{m-n}^y}
&\leq&\left(\sum_{|n|\geq M}\frac{1}{\meanv{n}^{xq}}\right)^{\frac1q}
\left(\sum_{\{|n|\geq\frac{|m|}{2}\}}\frac{1}{\meanv{n}^{\frac{yq}{q-1}}}\right)^{\frac{q-1}{q}}\nn\\
&\leq& \frac{c_1(q,x)}{|m|^{\frac1q-(1-y)}\langle M \rangle^{x-\frac1q}}\,.\nn
\eea
So we obtained (\ref{eq:conv1}), and (\ref{eq:conv0}) also follows taking $M=0$. 
\end{proof}

\begin{remark}\label{rmk:conv1}
The following particular cases of Lemma \ref{lemma:conv1} will be useful in the sequel.
\begin{itemize}
\item[i)] For $\frac12  < s  < \frac32$ we set
\be\label{eq:scelta-qp}
x-\frac1p=\frac1q-(1-y)=s-\frac12 
\ee
and get
\bea
\sum_{n\in\Z}\frac{1}{\meanv{n}^{2s-1}\meanv{m-n}}
&\leq &\frac{c}{\meanv{m}^{s-\frac12}}\label{eq:case-I-conv}\,,\\
\sum_{|n|\geq M}\frac{1}{\meanv{n}^{2s-1} \meanv{m-n}}&\leq& 
\frac{c}{\meanv{m}^{s-\frac12}}\left(1_{\{|m|\gtrsim M\}}+\frac{1}{\meanv{M}^{s-\frac12}}\right)
\,.\label{eq:case-I-convM}
\eea
\item[ii)] For $s \geq \frac32$ we set $p=q= 1$, which gives 
\bea
\sum_{n\in\Z}\frac{1}{\meanv{n}^{2s-1} \meanv{m-n}}&\leq &\frac{c}{\meanv{m}}\label{eq:case-I-conv3/2}\,,\\ 
\label{eq:case-I-conv3/2M}
\sum_{|n|\geq M}\frac{1}{ \meanv{n}^{2s-1} \meanv{m-n}}&\leq& 
 \frac{c}{\meanv{m}}
 \left(1_{\{|m|\gtrsim M\}} + \frac{1}{\meanv{M}}\right)  
\,.
\eea
\item[iii)] For $s > \frac12$ by the same choice (\ref{eq:scelta-qp}) we get
\bea
\sum_{n\in\Z}\frac{1}{\meanv{n}^s\meanv{m-n}^{s}}&\leq &\frac{c}{\meanv{m}^{s-\frac12}}\,,\label{eq:case-II-conv}\\
\sum_{|n|\geq M}\frac{1}{\meanv{n}^s\meanv{m-n}^{s}}&\leq &
\frac{c}{\meanv{m}^{s-\frac12}} 
\left(1_{\{|m|\gtrsim M\}}+\frac{1}{\meanv{M}^{s-\frac12}}\right)\,.\label{eq:case-II-convM}
\eea
\item[iv)] For $s > \frac12$ we set $\frac1p = \varepsilon, q=1$, where $\varepsilon >0$ 
may be chosen arbitrarily small, and get
\bea
\sum_{n\in\Z}\frac{1}{\meanv{n}^{2s+1}\meanv{m-n}}&\leq &\frac{c}{\meanv{m}}\,,\label{eq:case-IV-conv}\\
\sum_{|n|\geq M}\frac{1}{\meanv{n}^{2s+1}\meanv{m-n}}&\leq &
\frac{c}{\meanv{m}} 
\left(1_{\{|m|\gtrsim M\}}+\frac{1}{\meanv{M}^{2s}}\right)\,.\label{eq:case-IV-convM}
\eea
\end{itemize}
\end{remark}

Thanks to Proposition \ref{prop:growth-Hs-norm}, we can split the proof of Proposition \ref{prop:Wick} into three steps, one statement for each $F_{i,N}$, $i=1,2,3$ (recall \eqref{Def:F1N}, \eqref{Def:F2N} and \eqref{Def:F3N}). Again, 
we are using the  shorten notation $F_{i, N} = F_{i, N}(0,u)$.

We feel the need to warn the reader about the next somewhat lengthy computations. In particular the proof of subsequent Lemma \ref{Lemma:DecayF2} is a long enumeration of cases, each of which reduces to a term already estimated in the proof of Lemma \ref{Lemma:DecayF1}. The proofs of Lemma \ref{Lemma:DecayF1} and Lemma \ref{Lemma:Lemma:DecayF2} are similar, but independent one from each other. 
  
In what follows we use crucially the Wick formula for expectation values of multilinear forms of Gaussian random variables in the following form. 
Let $\ell \in \N$ and $S_{\ell}$ be the 
symmetric group on $\{1,\dots,\ell\}$, whose elements are denoted by $\s$. 
Recalling that $\hat u(-n) = \overline{ \hat{u}(n)}$  
\begin{equation}\label{eq:Wick}
E_s\Big[ \prod_{j=1}^{\ell} \hat  u(n_j) \hat  u(-m_j)  \Big]= E_s\Big[ \prod_{j=1}^{\ell} \hat  u(n_j)  \overline{ \hat{u}(m_j)}  \Big] 
= 
\sum_{\sigma \in S_{\ell}}\prod_{j=1}^{\ell} \frac{\d_{m_j,n_{\sigma(j)}}}{|n_j|^{2s+1}} 
 \, ,
\end{equation}
where $E_s$ is the expectation w.r.t. $\gamma_s$. In the following we shall use $\ell=3$ and often refer to the elements of $S_{3}$ as {\em contractions} (of indeces). 

The following set will appear int he next three proofs. Given a vector $a \in \mathbb{Z}^3$ we denote $a_j$ its components and
we define 
\begin{equation}\label{eq:A-NM}
A_{N,M}= 
\{a_1,a_2,a_3 \neq 0,\,
 a_1+a_2+a_3 =  0 \,,\,N\geq \max_j(|a_j|)>M\}\,.
\end{equation}
Note that if $a \in A_{N,M}$ we must have 
\begin{equation}\label{ZeroAverageWick}
a_j \neq -a_{j'} \quad \mbox{for all} \quad j, j' \in \{1,2,3 \}\,.
\end{equation}
This is because of the restrictions $a_j \neq 0$, $a_1+a_2+a_3 =  0$.

\begin{lemma}\label{Lemma:DecayF1}
For all $N>M$ it holds
\bea
\| F_{1,N} - F_{1,M} \|_{L^{2}(\gamma_s)} &\lesssim& \frac{1}{M^{\frac{s}{2}-\frac14}}\,,
\quad \quad s\in\left(\frac12,\frac32\right] \label{DecayF1-s-piccolo}\\
\| F_{1,N} - F_{1,M} \|_{L^{2}(\gamma_s)} &\lesssim& \frac{1}{\sqrt{M}}\,, \quad \quad s>\frac32\label{DecayF1-s-grande}\,.
\eea
\end{lemma}

\begin{proof}
We have (recall \eqref{Def:F1N})
$$
F_{1,N} - F_{1,M} = \frac{2 i}{\pi} \sum_{n \in A_{N.M}} |n_1|^{s} |n_2|^s n_3 \, \hat u(n_1) \hat  u(n_2) \hat  u(n_3)  \,.
$$
and taking the modulus squared (recall $\hat u(-m_j) = \overline{ \hat{u}(m_j)}$)
\begin{equation}\label{SquareBeforeWick}
|F_{1,N} - F_{1,M}|^2 =  \frac{4}{\pi^2} \sum_{(n,m) \in A_{N.M}^{2} } |n_1|^{s} |n_2|^s n_3 |m_1|^{s} |m_2|^s m_3 \prod_{j=1}^{3} \hat  u(n_j) \hat  u(-m_j) \, .
\end{equation}
When taking the expected value of \eqref{SquareBeforeWick} 
w.r.t. $\gamma_s$ we use the
Wick formula \eqref{eq:Wick} with~$\ell=3$. This gives 
\begin{equation}
\| F_{1,N} - F_{1,M} \|_{L^{2}(\gamma_s)}^2 = \frac{4}{\pi^2}
\sum_{\sigma \in S_3} \sum_{n \in A_{N,M}} \frac{ |n_1|^{s}|n_{\sigma(1)}|^s |n_2|^s |n_{\sigma(2)}|^s n_3 n_{\sigma(3)}  }{ |n_1|^{2s +1} |n_2|^{2s +1} |n_3|^{2s +1} }
\end{equation}
It is easy to see that the contractions $\sigma=(1,2,3)$ and $\sigma=(2,1,3)$ give the same contributions. Also, the remaining contractions give all the same contributions.
Thus we may reduce to the cases $\sigma=(1,2,3)$ and (say) $\sigma=(1,3,2)$. 

The contribution relative to $\sigma=(1,2,3)$ is
\be\label{1Bound}
\sum_{\substack{ |n_j| \leq N, n_j \neq 0 \\ n_1 + n_2 + n_3 = 0 \\ \max_j(|n_j|) >M} } \frac{ |n_1|^{2s} |n_2|^{2s}  n_3^2   }{ |n_1|^{2s +1} |n_2|^{2s +1} |n_3|^{2s +1} }
\leq
\sum_{\substack{ |n_j| \leq N, n_j \neq 0 \\ n_1 + n_2 + n_3 = 0 \\ \max_j(|n_j|) >M} } \frac{ |n_1|^{2s} |n_2|^{2s} | n_3|^2   }{ |n_1|^{2s +1} |n_2|^{2s +1} |n_3|^{2s +1} }\,.
\ee
We write

\bea
\mbox{r.h.s. of }\eqref{1Bound}&\lesssim&\sum_{\substack{ |n_j| \leq N \\ n_1 + n_2 + n_3 = 0 \\ \max_j(|n_j|) >M} } \frac{1}{ \meanv{n_1} \meanv{n_2} \meanv{n_3}^{2s -1} }\nn\\
&\leq&\sum_{\substack{ |n_2|, |n_3| \leq N \\ \max{(|n_2|,|n_3|)} \gtrsim M} } \frac{1}{ \meanv{n_2} \meanv{n_2-n_3} \meanv{n_3}^{2s -1} }\nn\\
&\lesssim&\sum_{|n_2|>M}\frac{1}{\meanv{n_2}}\sum_{n_3\in\Z } \frac{1}{\meanv{n_2-n_3} \meanv{n_3}^{2s -1} }\label{eq:1Bound-A}\\
&+&\sum_{n_2\in\Z}\frac{1}{\meanv{n_2}}\sum_{|n_3|>M } \frac{1}{\meanv{n_2-n_3} \meanv{n_3}^{2s -1} }\label{eq:1Bound-B}\,.
\eea
In the second inequality we used the symmetry of the r.h.s. under $m_3 \leftrightarrow -m_3$ and that
$n_1 + n_2 + n_3 = 0$ and $\max_j(|n_j|) >M$ imply $\max(|n_2|,|n_3|)\gtrsim M$.

For $s\in(\frac12,\frac32]$ the inner sums in (\ref{eq:1Bound-A}) and (\ref{eq:1Bound-B}) can be estimated respectively by (\ref{eq:case-I-conv}) and~(\ref{eq:case-I-convM}) 
\bea
\eqref{eq:1Bound-A}&\leq&
\sum_{|n_2|>M} \frac{1}{\meanv{n_2}} \frac{1}{\meanv{n_2}^{s-\frac12}} = 
\sum_{|n_2|>M}\frac{1}{\meanv{n_2}^{s+\frac12}}
\lesssim \frac{1}{M^{s-\frac12}}\,,\label{eq:1Bound-A-fine}\\
\eqref{eq:1Bound-B} &\leq& \sum_{|n_2|\gtrsim M}\frac{1}{\meanv{n_2}^{s+\frac12}}
+\frac{1}{M^{s-\frac12}}\sum_{n_2\in\Z}\frac{1}{\meanv{n_2}^{s+\frac12}}
\lesssim \frac{1}{M^{s-\frac12}}\,.\label{eq:1Bound-B-fine}
\eea
For $s>\frac32$ we estimate the inner sums of (\ref{eq:1Bound-A}) and (\ref{eq:1Bound-B}) 
using the inequalities 
(\ref{eq:case-I-conv3/2}) and (\ref{eq:case-I-conv3/2M}) 
and obtain
\bea
\eqref{eq:1Bound-A}&\leq&\sum_{|n_2|>M}\frac{1}{\meanv{n_2}^{2}}\lesssim \frac{1}{M}\,,\label{eq:1Bound-A-fine-sgrande}\\
\eqref{eq:1Bound-B}&\leq&\sum_{|n_2|\gtrsim M}\frac{1}{\meanv{n_2}^{2}}+\frac{1}{M}\sum_{n_2\in\Z}\frac{1}{\meanv{n_2}^{2}}\lesssim \frac{1}{M}\,.\label{eq:1Bound-B-fine-sgrande}
\eea

The contribution relative to $\sigma=(1,3,2)$ is 
 \begin{equation}\label{2Bound}
 \sum_{\substack{ |n_j| \leq N, n_j \neq 0 \\ \max_j(|n_j|) >M} } \frac{ |n_1|^{2s} |n_2|^{s} n_2  |n_3|^{s}n_3   }{ |n_1|^{2s +1} |n_2|^{2s +1} |n_3|^{2s +1} }
\leq
\sum_{\substack{ |n_j| \leq N, n_j \neq 0 \\ \max_j(|n_j|) >M} } \frac{ |n_1|^{2s} |n_2|^{s+1}  |n_3|^{s+1}   }{ |n_1|^{2s +1} |n_2|^{2s +1} |n_3|^{2s +1} }
\end{equation}
Again we can write
\bea
\mbox{r.h.s. of }\eqref{2Bound}&\lesssim&\sum_{\substack{ |n_j| \leq N \\ n_1 + n_2 + n_3 = 0 \\ \max_j(|n_j|) >M} } \frac{1}{ \meanv{n_1} \meanv{n_2}^s \meanv{n_3}^{s} }\nn\\
&=&\sum_{\substack{ |n_2|, |n_3| \leq N \\ \max{(|n_1|,|n_2|)} >M} } \frac{1}{ \meanv{n_1}\meanv{n_2}^{s} \meanv{n_1-n_2}^s  }\nn\\
&\leq&\sum_{|n_1| \gtrsim M}\frac{1}{\meanv{n_1}}\sum_{n_2\in\Z } \frac1{\meanv{n_2}^{s} \meanv{n_1-n_2}^s}\label{eq:2Bound-A}\\
&+&\sum_{n_1\in\Z}\frac{1}{\meanv{n_1}}\sum_{|n_2| \gtrsim M } \frac1{\meanv{n_2}^{s} \meanv{n_1-n_2}^s}\label{eq:2Bound-B}\,.
\eea
The inner sums in (\ref{eq:2Bound-A}) and (\ref{eq:2Bound-B}) are estimated respectively by (\ref{eq:case-II-conv}) and (\ref{eq:case-II-convM}) and we obtain
\bea
\eqref{eq:2Bound-A}&\leq&\sum_{|n_1|>M}\frac{1}{\meanv{n_1}^{s+\frac12}} 
\lesssim \frac{1}{M^{s-\frac12}}\,,\label{eq:2Bound-A-fine}\\
\eqref{eq:2Bound-B}&\leq&\sum_{|n_1|\gtrsim M}\frac{1}{\meanv{n_1}^{s+\frac12}}
+\frac{1}{M^{s-\frac12}}\sum_{n_1\in\Z}\frac{1}{\meanv{n_1}^{s+ \frac12}}\lesssim \frac{1}{M^{s-\frac12}}\,.\label{eq:2Bound-B-fine}
\eea
\end{proof}

\begin{lemma}\label{Lemma:DecayF2}
For all $N>M$ it holds
\bea
\| F_{2,N} - F_{2,M} \|_{L^{2}(\gamma_s)} &\lesssim& \frac{1}{M^{\frac{s}{2}-\frac14}}\,,
\quad \quad s\in\left(\frac12,\frac32\right] \label{DecayF1-s-piccolo}\\
\| F_{2,N} - F_{2,M} \|_{L^{2}(\gamma_s)} &\lesssim& \frac{1}{\sqrt{M}}\,, \quad \quad s>\frac32\label{DecayF1-s-grande}\,.
\eea
\end{lemma}

\begin{proof}
In fact, we will reduce to a sum of contributions which are the same as the ones handled in the previous 
lemma.

Recalling that 
$$
\big[ |D_x|^{s}, P_N u(t)  \big] \partial_x P_Nu(t) = |D_x|^{s} ( P_N u(t) \partial_x P_N u(t))  -  P_N u(t) |D_x^s|  \partial_x P_N u(t)    
$$
and proceeding as in the proof of Lemma \ref{Lemma:DecayF1} we obtain
$$
F_{2,N} - F_{2,M} = \frac{4i}{\pi} \sum_{n \in A_{N.M}} |n_1|^s n_2 (|n_2 + n_3|^s - |n_2|^s) \, \hat u(n_1) \hat  u(n_2) \hat  u(n_3)   \,.
$$
Taking the modulus squared
\begin{equation}\label{SquareBeforeWick2}
|F_{2,N} - F_{2,M}|^2 =  \frac{16}{\pi^2} \sum_{(n,m) \in A_{N.M}^{2} } 
|n_1|^s n_2 (|n_2 + n_3|^s - |n_2|^s)
|m_1|^s m_2 (|m_2 + m_3|^s - |m_2|^s)  \prod_{j=1}^{3} \hat  u(n_j) \hat  u(-m_j) 
\end{equation}
and using 
the
Wick formula \eqref{eq:Wick} with~$\ell=3$ we arrive to 
\begin{equation}\label{WRTDTHis}
\| F_{2,N} - F_{2,M} \|_{L^{2}(\gamma_s)}^2 = \frac{16}{\pi^2}
\sum_{\sigma \in S_3} \sum_{n \in A_{N,M}} 
\frac{ |n_1|^s |n_{\sigma(1)}|^s n_2 n_{\sigma(2)} 
(|n_2 + n_3|^s - |n_2|^s)  (|n_{\sigma(2)} + n_{\sigma(3)}|^s - |n_{\sigma(2)}|^s)     }{ |n_1|^{2s +1} |n_2|^{2s +1} |n_3|^{2s +1} }\,.
\end{equation}

Before evaluating all the contributions relative to the different $\sigma$, we recall an useful inequality to handle the difference 
$$
| |a + b|^s - |a|^s |\,.
$$
We distinguish two cases, namely
$$
|a| \leq 2 |b| , \quad  |a| > 2 |b| \,.
$$
In the first case we have immediately
\begin{equation}\label{EasysMainBound}
| |a + b|^s - |a|^s | \lesssim |b|^s, \qquad \mbox{for $|a| \leq 2 |b|$} \, .
\end{equation}
In the second case 
we  
use, for $s>0$, the Taylor expansion (converging for~$|x|<1$)
\be\nonumber
(1+x)^s=1 + \sum_{k\geq1}\frac{(s)_k}{k!}x^k\,,
\ee
where $(s)_k$ is defined by
$$
(s)_0=1\,,\quad (s)_{k}:=\prod_{j=0}^{k-1}(s-j)\,,\quad k\geq1\,. 
$$ 
Letting $x :=  \frac{|b|}{|a|} < \frac{1}{2} $ and using $|(s)_k| \leq k!$ we can bound
\begin{equation}\label{HardsMainBound}
| |a + b|^s - |a|^s | = |a|^s \left|  \sum_{k\geq 1} \left(\frac{|b|}{|a|}\right)^k  \right|  \lesssim |a|^{s-1} |b|, \qquad \mbox{for $|a| > 2 |b|$}  \,.
\end{equation}

Now we are ready to estimate \eqref{WRTDTHis}.

$\bullet$ Permutation $\sigma =(1,2,3)$. We need to handle
\begin{equation}\nonumber
\sum_{\substack{ |n_j| \leq N, n_j \neq 0 \\ n_1 + n_2 + n_3 = 0 \\ \max_j(|n_j|) >M} } 
\frac{  |n_1|^{2s} n_2^2  
(|n_2 + n_3|^s - |n_2|^s)^2       }{ |n_1|^{2s +1} |n_2|^{2s +1} |n_3|^{2s +1} }\,. 
\end{equation}
If $|n_2| \leq 2|n_3|$ we can use \eqref{EasysMainBound} to bound $\big| |n_2 + n_3|^s - |n_2|^s \big| \lesssim |n_3|^s $ so that
\begin{equation}\nonumber
\sum_{\substack{ |n_j| \leq N, n_j \neq 0 \\ n_1 + n_2 + n_3 = 0 \\ \max_j(|n_j|) >M \\ |n_2| \leq 2 |n_3|} } 
\frac{  |n_1|^{2s} n_2^2  
|n_3|^{2s}       }{ |n_1|^{2s +1} |n_2|^{2s +1} |n_3|^{2s +1} }
\leq 
\sum_{\substack{ |n_j| \leq N, n_j \neq 0 \\ n_1 + n_2 + n_3 = 0 \\ \max_j(|n_j|) >M }} 
\frac{  |n_1|^{2s} |n_2|^2  
|n_3|^{2s}       }{ |n_1|^{2s +1} |n_2|^{2s +1} |n_3|^{2s +1} }.
\end{equation}
This is done as \eqref{1Bound} (exchanging $n_2 \leftrightarrow n_3$).

If $|n_2| >  2|n_3|$ we  use \eqref{HardsMainBound} to bound $\big| |n_2 + n_3|^s - |n_2|^s \big| \lesssim |n_2|^{s-1}|n_3| $ and we reduce to estimate 
\begin{equation}\nonumber
\sum_{\substack{ |n_j| \leq N, n_j \neq 0 \\ n_1 + n_2 + n_3 = 0 \\ \max_j(|n_j|) >M \\ |n_2| > 2 |n_3|} } 
\frac{  |n_1|^{2s} |n_2|^2  
|n_2|^{2s-2}|n_3|^2    }{ |n_1|^{2s +1} |n_2|^{2s +1} |n_3|^{2s +1} }
\leq \sum_{\substack{ |n_j| \leq N, n_j \neq 0 \\ n_1 + n_2 + n_3 = 0 \\ \max_j(|n_j|) >M } } 
\frac{  |n_1|^{2s} |n_2|^{2s}  
|n_3|^{2}       }{ |n_1|^{2s +1} |n_2|^{2s +1} |n_3|^{2s +1} }.
\end{equation}
We can again proceed as we have done for \eqref{1Bound}, getting the same decay rate.

$\bullet$ Permutation $\sigma =(1,3,2)$. We need to handle
\begin{equation}\label{2BoundBis}
\sum_{\substack{ |n_j| \leq N, n_j \neq 0 \\ n_1 + n_2 + n_3 = 0 \\ \max_j(|n_j|) >M} } 
\frac{  |n_1|^{2s}  n_2 n_3 ( |n_2 + n_3|^s - |n_2|^s) ( |n_2 + n_3|^s - |n_3|^s )
 }{ |n_1|^{2s +1} |n_2|^{2s +1} |n_3|^{2s +1} }\,.
\end{equation}
We have three possibilities: 
\begin{enumerate}[(A)]
\item  $|n_2| \leq 2 |n_3|$ and $|n_3| \leq 2 |n_2|$, 
\item   $|n_3| > 2 |n_2|$, 
\item   $|n_2| > 2|n_3|$.
\end{enumerate}
In the case (A) we use \eqref{EasysMainBound} to bound 
$\big| |n_2 + n_3|^s - |n_2|^s\big| \lesssim |n_3|^s$ and $| |n_2 + n_3|^s - |n_3|^s \big| \lesssim |n_2|^s$. Thus
 we reduce to estimate
\begin{equation}\nonumber
\sum_{\substack{ |n_j| \leq N, n_j \neq 0 \\ n_1 + n_2 + n_3 = 0 \\ \max_j(|n_j|) >M \\ |n_2| \leq 2 |n_3| \, \mbox{and} \, |n_3| \leq 2 |n_2|} } 
\frac{  |n_1|^{2s}  n_2 n_3 |n_3|^s |n_2|^s
 }{ |n_1|^{2s +1} |n_2|^{2s +1} |n_3|^{2s +1} }
\leq
\sum_{\substack{ |n_j| \leq N, n_j \neq 0 \\ n_1 + n_2 + n_3 = 0 \\ \max_j(|n_j|) >M }} 
\frac{  |n_1|^{2s}  |n_2|^{s+1} |n_3|^{s+1} 
 }{ |n_1|^{2s +1} |n_2|^{2s +1} |n_3|^{2s +1} } \, .
\end{equation}
This is done as \eqref{2Bound}. 

If we are in the case (B) we have in particular $|n_2| < \frac{1}{2} |n_3| \leq 2 |n_3|$, so we can use \eqref{EasysMainBound} 
to bound the difference $||n_2 + n_3|^s - |n_2|^s| \lesssim |n_3|^s$ 
and \eqref{HardsMainBound} 
to bound the difference $||n_2 + n_3|^s - |n_3|^s| \lesssim |n_3|^{s-1} |n_2|$. Thus we need to estimate
\begin{equation}\nonumber
\sum_{\substack{ |n_j| \leq N, n_j \neq 0 \\ n_1 + n_2 + n_3 = 0 \\ \max_j(|n_j|) >M \\ |n_3| > 2 |n_2|}  } 
\frac{  |n_1|^{2s}  n_2 n_3 |n_3|^s |n_3|^{s-1} |n_2| 
 }{ |n_1|^{2s +1} |n_2|^{2s +1} |n_3|^{2s +1} }
\leq
\sum_{\substack{ |n_j| \leq N, n_j \neq 0 \\ n_1 + n_2 + n_3 = 0 \\ \max_j(|n_j|) >M }  } 
\frac{  |n_1|^{2s}  |n_2|^{2}   |n_3|^{2s} 
 }{ |n_1|^{2s +1} |n_2|^{2s +1} |n_3|^{2s +1} }
\end{equation}
and this is done as \eqref{1Bound}. The case (C) is the same as (B) exchanging $n_2 \leftrightarrow n_3$.

$\bullet$ Permutation $\sigma =(2,1,3)$. We need to handle
\begin{equation}\nonumber
\sum_{\substack{ |n_j| \leq N, n_j \neq 0 \\ n_1 + n_2 + n_3 = 0 \\ \max_j(|n_j|) >M} } 
\frac{ |n_1|^s |n_{2}|^s n_2 n_{1} 
(|n_2 + n_3|^s - |n_2|^s)  (|n_{1} + n_{3}|^s - |n_{1}|^s)     }{ |n_1|^{2s +1} |n_2|^{2s +1} |n_3|^{2s +1} }\,.
\end{equation}
We distinguish
\begin{enumerate}[(A)]
\item $|n_2| \leq 2 |n_3|$, $|n_1| \leq 2 |n_3|$
\item $|n_2| \leq 2 |n_3|$, $|n_1| > 2 |n_3|$
\item $|n_2| > 2 |n_3|$, $|n_1| \leq 2 |n_3|$;  \qquad (same as (B) switching $n_2 \leftrightarrow n_1$)
\item $|n_2| > 2 |n_3|$, $|n_1| > 2 |n_3|$\,.
\end{enumerate}
In the case (A) we use \eqref{EasysMainBound} to bound $ \big| |n_2 + n_3|^s - |n_2|^s\big| \lesssim |n_3|^s$ and
 $ \big| |n_{1} + n_{3}|^s - |n_{1}|^s \big| \lesssim |n_3|^s$. Thus we need to estimate  
\begin{equation}\nonumber
\sum_{\substack{ |n_j| \leq N, n_j \neq 0 \\ n_1 + n_2 + n_3 = 0 \\ \max_j(|n_j|) >M \\ |n_2| \leq 2 |n_3| \, \mbox{and} \, |n_1| \leq 2 |n_3|} } 
\frac{ |n_1|^s |n_{2}|^s n_2 n_{1} 
|n_3|^s |n_3|^s     }{ |n_1|^{2s +1} |n_2|^{2s +1} |n_3|^{2s +1} }
\leq
\sum_{\substack{ |n_j| \leq N, n_j \neq 0 \\ n_1 + n_2 + n_3 = 0 \\ \max_j(|n_j|) >M } } 
\frac{ |n_1|^{s+1}  |n_{2}|^{s+1}   
  |n_{3}|^{2s}     }{ |n_1|^{2s +1} |n_2|^{2s +1} |n_3|^{2s +1} }\,,
\end{equation}
that is done as \eqref{2Bound}.

In the case (B) we use  \eqref{EasysMainBound}  to bound 
$\big||n_2 + n_3|^s - |n_2|^s\big| \lesssim |n_3|^s$ and \eqref{HardsMainBound} 
to bound $\big| |n_{1} + n_{3}|^s - |n_{1}|^s\big| \lesssim |n_1|^{s-1} |n_3|$, so that
 \begin{equation}\nonumber
\sum_{\substack{ |n_j| \leq N, n_j \neq 0 \\ n_1 + n_2 + n_3 = 0 \\ \max_j(|n_j|) >M \\ |n_2| \leq 2 |n_3| \, \mbox{and} \, |n_1| > 2 |n_3|} } 
\frac{ |n_1|^s |n_{2}|^s n_2 n_{1} 
|n_3|^s  |n_1|^{s-1} |n_3|    }{ |n_1|^{2s +1} |n_2|^{2s +1} |n_3|^{2s +1} }
\leq
\sum_{\substack{ |n_j| \leq N, n_j \neq 0 \\ n_1 + n_2 + n_3 = 0 \\ \max_j(|n_j|) >M } } 
\frac{ |n_1|^{2s} |n_{2}|^{s+1}   
|n_3|^{s+1}         }{ |n_1|^{2s +1} |n_2|^{2s +1} |n_3|^{2s +1} }\,,
\end{equation}
that is estimated as \eqref{2Bound}. The case (C) is the same as (B) using the symmetry $n_2 \leftrightarrow n_1$. 

In the case (D) we 
 use
 \eqref{HardsMainBound}  to bound 
$\big||n_2 + n_3|^s - |n_2|^s\big| \lesssim |n_2|^{s-1} |n_3|$ and  
$\big| |n_{1} + n_{3}|^s - |n_{1}|^s\big| \lesssim |n_1|^{s-1} |n_3|$, so that
\begin{equation}\nonumber
\sum_{\substack{ |n_j| \leq N, n_j \neq 0 \\ n_1 + n_2 + n_3 = 0 \\ \max_j(|n_j|) >M \\ |n_2| > 2 |n_3| \, \mbox{and} \, |n_1| > 2 |n_3|} } 
\frac{ |n_1|^s |n_{2}|^s n_2 n_{1} 
|n_2|^{s-1} |n_3|    |n_1|^{s-1} |n_3|     }{ |n_1|^{2s +1} |n_2|^{2s +1} |n_3|^{2s +1} }
\leq
\sum_{\substack{ |n_j| \leq N, n_j \neq 0 \\ n_1 + n_2 + n_3 = 0 \\ \max_j(|n_j|) >M } } 
\frac{ |n_1|^{2s} |n_{2}|^{2s}|n_{3}|^{2}  }{ |n_1|^{2s +1} |n_2|^{2s +1} |n_3|^{2s +1} }\,,
\end{equation}
that is estimated as \eqref{1Bound}.  

$\bullet$ Permutation $\sigma =(2,3,1)$. We need to handle
\begin{equation}\nonumber
\sum_{\substack{ |n_j| \leq N, n_j \neq 0 \\ n_1 + n_2 + n_3 = 0 \\ \max_j(|n_j|) >M} } 
 \frac{ |n_1|^s |n_{2}|^s n_2 n_{3} 
(|n_2 + n_3|^s - |n_2|^s)  (|n_{3} + n_{1}|^s - |n_{3}|^s)     }{ |n_1|^{2s +1} |n_2|^{2s +1} |n_3|^{2s +1} }\,.
\end{equation}
We distinguish
\begin{enumerate}[(A)]
\item $|n_2| \leq 2 |n_3|$, $|n_3| \leq 2 |n_1|$;
\item $|n_2| \leq 2 |n_3|$, $|n_3| > 2 |n_1|$;
\item $|n_2| > 2 |n_3|$, $|n_3| \leq 2 |n_1|$;    
\item $|n_2| > 2 |n_3|$, $|n_3| > 2 |n_1|$.
\end{enumerate}

In the case (A) 
we use \eqref{EasysMainBound}  to bound $\big| |n_2 + n_3|^s - |n_2|^s \big| \lesssim |n_3|^s$
and  $\big| |n_{3} + n_{1}|^s - |n_{3}|^s\big|\lesssim  |n_1|^s $ so that
\begin{equation}\nonumber
\sum_{\substack{ |n_j| \leq N, n_j \neq 0 \\ n_1 + n_2 + n_3 = 0 \\ \max_j(|n_j|) >M \\ |n_2| \leq 2 |n_3| \, \mbox{and} \, |n_3| \leq 2 |n_1|} } 
 \frac{ |n_1|^s |n_{2}|^s n_2 n_{3} 
|n_3|^s |n_1|^s     }{ |n_1|^{2s +1} |n_2|^{2s +1} |n_3|^{2s +1} }
\leq \sum_{\substack{ |n_j| \leq N, n_j \neq 0 \\ n_1 + n_2 + n_3 = 0 \\ \max_j(|n_j|) >M } } 
 \frac{ |n_1|^{2s} |n_{2}|^{s+1} |n_{3}|^{s + 1} 
      }{ |n_1|^{2s +1} |n_2|^{2s +1} |n_3|^{2s +1} } \, ,
\end{equation}
that is estimated as \eqref{2Bound}.

In the case (B) 
we use \eqref{EasysMainBound} to bound $\big| |n_2 + n_3|^s - |n_2|^s \big| \lesssim |n_3|^s$
and \eqref{HardsMainBound}  to bound $\big| |n_{3} + n_{1}|^s - |n_{3}|^s\big|\lesssim |n_{3}|^{s-1} |n_1| $, thus
\begin{equation}\nonumber
\sum_{\substack{ |n_j| \leq N, n_j \neq 0 \\ n_1 + n_2 + n_3 = 0 \\ \max_j(|n_j|) >M \\ |n_2| \leq 2 |n_3| \, \mbox{and} \, |n_3| > 2 |n_1|} } 
 \frac{ |n_1|^s |n_{2}|^s n_2 n_{3} 
|n_3|^s |n_3|^{s-1} |n_1|      }{ |n_1|^{2s +1} |n_2|^{2s +1} |n_3|^{2s +1} }
\leq \sum_{\substack{ |n_j| \leq N, n_j \neq 0 \\ n_1 + n_2 + n_3 = 0 \\ \max_j(|n_j|) >M } } 
 \frac{ |n_1|^{s+1} |n_{2}|^{s+1} |n_{3}|^{2s} 
      }{ |n_1|^{2s +1} |n_2|^{2s +1} |n_3|^{2s +1} } \, ,
\end{equation}
that is again estimated as \eqref{2Bound}.

In the case (C) 
we use \eqref{HardsMainBound}  to bound $\big| |n_2 + n_3|^s - |n_2|^s \big| \lesssim |n_2|^{s-1} |n_3|$
and \eqref{EasysMainBound}  to bound $\big| |n_{3} + n_{1}|^s - |n_{3}|^s\big|\lesssim  |n_1|^s $, thus
\begin{equation}\nonumber
\sum_{\substack{ |n_j| \leq N, n_j \neq 0 \\ n_1 + n_2 + n_3 = 0 \\ \max_j(|n_j|) >M \\ |n_2| > 2 |n_3| \, \mbox{and} \, |n_3| \leq 2 |n_1|} } 
 \frac{ |n_1|^s |n_{2}|^s n_2 n_{3} 
|n_2|^{s-1} |n_3| |n_1|^s      }{ |n_1|^{2s +1} |n_2|^{2s +1} |n_3|^{2s +1} }
\leq \sum_{\substack{ |n_j| \leq N, n_j \neq 0 \\ n_1 + n_2 + n_3 = 0 \\ \max_j(|n_j|) >M } } 
 \frac{ |n_1|^{2s} |n_{2}|^{2s} |n_{3}|^{2} 
      }{ |n_1|^{2s +1} |n_2|^{2s +1} |n_3|^{2s +1} } \, ,
\end{equation}
that is  estimated as \eqref{1Bound}.

In the case (D) 
we use \eqref{HardsMainBound}  to bound $\big| |n_2 + n_3|^s - |n_2|^s \big| \lesssim |n_2|^{s-1} |n_3|$
and $\big| |n_{3} + n_{1}|^s - |n_{3}|^s\big|\lesssim  |n_3|^{s-1} |n_1| $. We arrive to 
\begin{equation}\nonumber
\sum_{\substack{ |n_j| \leq N, n_j \neq 0 \\ n_1 + n_2 + n_3 = 0 \\ \max_j(|n_j|) >M \\ |n_2| > 2 |n_3| \, \mbox{and} \, |n_3| > 2 |n_1|} } 
 \frac{ |n_1|^s |n_{2}|^s n_2 n_{3} 
|n_2|^{s-1} |n_3| |n_3|^{s-1}  |n_1|      }{ |n_1|^{2s +1} |n_2|^{2s +1} |n_3|^{2s +1} }
\leq \sum_{\substack{ |n_j| \leq N, n_j \neq 0 \\ n_1 + n_2 + n_3 = 0 \\ \max_j(|n_j|) >M } } 
 \frac{ |n_1|^{s + 1} |n_{2}|^{2s} |n_{3}|^{s+1} 
      }{ |n_1|^{2s +1} |n_2|^{2s +1} |n_3|^{2s +1} } \, ,
\end{equation}
that is again estimated as \eqref{2Bound}.

$\bullet$ Permutation $\sigma =(3,2,1)$. We need to handle
\begin{equation}\nonumber
\sum_{\substack{ |n_j| \leq N, n_j \neq 0 \\ n_1 + n_2 + n_3 = 0 \\ \max_j(|n_j|) >M} } 
\frac{ |n_1|^s |n_{3}|^s n_2^2  
(|n_2 + n_3|^s - |n_2|^s)  (|n_{2} + n_{1}|^s - |n_{2}|^s)     }{ |n_1|^{2s +1} |n_2|^{2s +1} |n_3|^{2s +1} }
\end{equation}

We distinguish
\begin{enumerate}[(A)]
\item $|n_2| \leq 2 |n_3|$, $|n_2| \leq 2 |n_1|$;
\item $|n_2| \leq 2 |n_3|$, $|n_2| > 2 |n_1|$;
\item $|n_2| > 2 |n_3|$, $|n_2| \leq 2 |n_1|$;    \qquad (same as (B) switching $n_1 \leftrightarrow n_3$)
\item $|n_2| > 2 |n_3|$, $|n_2| > 2 |n_1|$.
\end{enumerate}

In the case (A) 
we use \eqref{EasysMainBound}  to bound $\big| |n_2 + n_3|^s - |n_2|^s \big| \lesssim |n_3|^{s} $
and $\big| |n_{2} + n_{1}|^s - |n_{2}|^s \big|\lesssim  |n_1|^{s} $. We arrive to 
\begin{equation}\nonumber
\sum_{\substack{ |n_j| \leq N, n_j \neq 0 \\ n_1 + n_2 + n_3 = 0 \\ \max_j(|n_j|) >M \\ |n_2| \leq 2 |n_3| \, \mbox{and} \, |n_2| \leq 2 |n_1|} } 
\frac{ |n_1|^s |n_{3}|^s n_2^2  
|n_3|^s  |n_{1}|^s      }{ |n_1|^{2s +1} |n_2|^{2s +1} |n_3|^{2s +1} }
\leq \sum_{\substack{ |n_j| \leq N, n_j \neq 0 \\ n_1 + n_2 + n_3 = 0 \\ \max_j(|n_j|) >M } } 
\frac{ |n_1|^{2s} |n_2|^2 |n_{3}|^{2s}      }{ |n_1|^{2s +1} |n_2|^{2s +1} |n_3|^{2s +1} } \, ,
\end{equation}
that is estimated as \eqref{1Bound}.

In the case (B) 
we use \eqref{EasysMainBound}  to bound $\big| |n_2 + n_3|^s - |n_2|^s \big| \lesssim |n_3|^{s} $
and  \eqref{HardsMainBound} to bound $\big| |n_{2} + n_{1}|^s - |n_{2}|^s \big|\lesssim |n_{2}|^{s-1} |n_1| $. We arrive to 
\begin{equation}\nonumber
\sum_{\substack{ |n_j| \leq N, n_j \neq 0 \\ n_1 + n_2 + n_3 = 0 \\ \max_j(|n_j|) >M \\ |n_2| \leq 2 |n_3| \, \mbox{and} \, |n_2| > 2 |n_1|} } 
\frac{ |n_1|^s |n_{3}|^s n_2^2  
|n_3|^s  |n_{2}|^{s-1} |n_1|      }{ |n_1|^{2s +1} |n_2|^{2s +1} |n_3|^{2s +1} }
\leq \sum_{\substack{ |n_j| \leq N, n_j \neq 0 \\ n_1 + n_2 + n_3 = 0 \\ \max_j(|n_j|) >M } } 
\frac{ |n_1|^{s+1} |n_2|^{s+1} |n_{3}|^{2s}      }{ |n_1|^{2s +1} |n_2|^{2s +1} |n_3|^{2s +1} } \, ,
\end{equation}
that is  estimated as \eqref{2Bound}.

The case (C) is the same as (B) exchanging $n_1 \leftrightarrow n_3$. 

In the case (D) 
we use \eqref{HardsMainBound}  to bound $\big| |n_2 + n_3|^s - |n_2|^s \big| \lesssim |n_2|^{s-1} |n_3| $
and $\big| |n_{2} + n_{1}|^s - |n_{2}|^s \big|\lesssim  |n_2|^{s-1}  |n_1| $. We arrive to 
\begin{equation}\nonumber
\sum_{\substack{ |n_j| \leq N, n_j \neq 0 \\ n_1 + n_2 + n_3 = 0 \\ \max_j(|n_j|) >M \\ |n_2| > 2 |n_3| \, \mbox{and} \, |n_2| > 2 |n_1|} } 
\frac{ |n_1|^s |n_{3}|^s n_2^2  
|n_2|^{s-1} |n_3| |n_2|^{s-1} |n_1|      }{ |n_1|^{2s +1} |n_2|^{2s +1} |n_3|^{2s +1} }
\leq \sum_{\substack{ |n_j| \leq N, n_j \neq 0 \\ n_1 + n_2 + n_3 = 0 \\ \max_j(|n_j|) >M } } 
\frac{ |n_1|^{s+1} |n_2|^{2s} |n_{3}|^{s+1}      }{ |n_1|^{2s +1} |n_2|^{2s +1} |n_3|^{2s +1} } \, ,
\end{equation}
that is estimated as \eqref{2Bound}.

$\bullet$ Permutation $\sigma =(3,1,2)$. We need to handle
\begin{equation}\nonumber
\sum_{\substack{ |n_j| \leq N, n_j \neq 0 \\ n_1 + n_2 + n_3 = 0 \\ \max_j(|n_j|) >M} } 
\frac{ |n_1|^s |n_{3}|^s n_2 n_{1} 
(|n_2 + n_3|^s - |n_2|^s)  (|n_{1} + n_{2}|^s - |n_{1}|^s)     }{ |n_1|^{2s +1} |n_2|^{2s +1} |n_3|^{2s +1} }
\end{equation}
We note that renaming the indeces $(n_1, n_3, n_2)$ with $(n_2,n_1,n_3)$ this reduces to
$$
\sum_{\substack{ |n_j| \leq N, n_j \neq 0 \\ n_1 + n_2 + n_3 = 0 \\ \max_j(|n_j|) >M} } 
 \frac{ |n_{2}|^s |n_1|^s   n_{3} n_2 
  (|n_{3} + n_{1}|^s - |n_{3}|^s)  (|n_2 + n_3|^s - |n_2|^s)     }{ |n_1|^{2s +1} |n_2|^{2s +1} |n_3|^{2s +1} },
$$
that is the contribution of the permutation $\sigma = (2,3,1)$. This completes the proof.
\end{proof}

\begin{lemma}\label{Lemma:Lemma:DecayF2}
Let $s >\frac12$. We have for all $N >M$
\begin{equation}\label{DecayF3}
\| F_{3,N} - F_{3,M} \|_{L^{2}(\gamma_s)} \lesssim \frac{1}{\sqrt{M}}\,.
\end{equation}
\end{lemma}

\begin{proof}
We have
$$
F_{3,N} - F_{3,M} = \frac{2i}{\pi} \sum_{n \in A_{N.M}} |n_1|^s \frac{(n_2+n_3)|n_2+n_3|^s}{1 + |n_2+n_3|}  \, u(n_1)  u(n_2) u(n_3)   \,.
$$
Taking the modulus squared
\begin{equation}\label{SquareBeforeWick3}
|F_{3,N} - F_{3,M}|^2 =  \frac{4}{\pi^2} \sum_{(n,m) \in A_{N,M}^{2} } 
|n_1|^s \frac{(n_2+n_3)|n_2+n_3|^s}{1 + |n_2+n_3|} 
|m_1|^s \frac{(m_2+m_3)|m_2+m_3|^s}{1 + |m_2+m_3|}   \prod_{j=1}^{3}u(n_j)u(-m_j) 
\end{equation}
and using 
the
Wick formula \eqref{eq:Wick} with~$\ell=3$ we arrive to 
\begin{align}\nonumber
\| F_{3,N} - F_{3,M} \|_{L^{2}(\gamma_s)}^2 &= \frac{4}{\pi^2}
\sum_{\sigma \in S_3} \sum_{n \in A_{N,M}} 
\frac{ |n_1|^s |n_{\sigma(1)}|^s \frac{(n_2+n_3)|n_2+n_3|^s}{1 + |n_2+n_3|} \frac{(n_{\sigma(2)}+n_{\sigma(3)})|n_{\sigma(2)}+n_{\sigma(3)}|^s}{1 + |n_{\sigma(2)}+n_{\sigma(3)}|}   }{ |n_1|^{2s +1} |n_2|^{2s +1} |n_3|^{2s +1} }
\\    \label{WRTDTHis3}
&
\leq \frac{4}{\pi^2} 
\sum_{\sigma \in S_3} \sum_{n \in A_{N,M}} 
\frac{ |n_1|^s |n_{\sigma(1)}|^s |n_2+n_3|^s |n_{\sigma(2)}+n_{\sigma(3)}|^s   }{ |n_1|^{2s +1} |n_2|^{2s +1} |n_3|^{2s +1} }.
\end{align}

It is easy to see that the contractions $\sigma=(1,2,3)$ and $\sigma=(1,3,2)$ give the same contributions. Also, the remaining contractions gives all the same contributions.
Thus we may reduce to the cases (say) $\sigma=(1,2,3)$ and $\sigma=(2,1,3)$. 

The contribution relative to $\sigma=(1,2,3)$ is
\begin{align}\nonumber
\sum_{\substack{ |n_j| \leq N, n_j \neq 0 \\ n_1 + n_2 + n_3 = 0 \\ \max_j(|n_j|) >M} }  &
\frac{ |n_1|^{2s}  |n_2+n_3|^{2s}  }{ |n_1|^{2s +1} |n_2|^{2s +1} |n_3|^{2s +1} }
\\ 
& \lesssim_s
\sum_{\substack{ |n_j| \leq N, n_j \neq 0 \\ n_1 + n_2 + n_3 = 0 \\ \max_j(|n_j|) >M} } 
\frac{ |n_1|^{2s}   |n_2|^{2s}  }{ |n_1|^{2s +1} |n_2|^{2s +1} |n_3|^{2s +1} }
+
\sum_{\substack{ |n_j| \leq N, n_j \neq 0 \\ n_1 + n_2 + n_3 = 0 \\ \max_j(|n_j|) >M} } 
\frac{ |n_1|^{2s}   |n_3|^{2s}  }{ |n_1|^{2s +1} |n_2|^{2s +1} |n_3|^{2s +1} }\nn\\
\end{align}
By the symmetry $n_2 \leftrightarrow n_3$ it suffices to handle the first term on the r.h.s., for which we have
\begin{align}
&\lesssim_s \sum_{\substack{ |n_1|,|n_2| \leq N  \\ |n_1| \gtrsim M} } 
\frac{ 1}{ \meanv{n_1}^{2s+1}\meanv{n_2} \meanv{n_2-n_1} }+\sum_{\substack{ |n_1|,|n_2| \leq N  \\ |n_2| \gtrsim M} } 
\frac{ 1}{ \meanv{n_1}^{2s+1}\meanv{n_2} \meanv{n_2-n_1} }\nn\\
&\lesssim_s \sum_{|n_2| \gtrsim M}\frac{1}{\meanv{n_2}}\sum_{n_1\in\Z } \frac{1}{ \meanv{n_1}^{2s+1}\meanv{n_2-n_1}}
\label{1BoundCaseF3}\\
&\qquad \qquad \qquad  + \sum_{n_2\in\Z}\frac{1}{\meanv{n_2} }\sum_{\substack{ |n_1| \gtrsim M} } \frac{1}{ \meanv{n_1}^{2s+1}\meanv{n_2-n_1} }
\,. \label{1BoundCaseF3-II}
\end{align}
The inner sums of (\ref{1BoundCaseF3}) and (\ref{1BoundCaseF3-II}) are estimated by 
(\ref{eq:case-IV-conv}) and (\ref{eq:case-IV-convM}). We have
\bea
\eqref{1BoundCaseF3}&\leq&\sum_{\substack{ |n_2| >M} } \frac{1}{ \meanv{n_2}^{2}}
\lesssim \frac{1}{M}\,,\\
\eqref{1BoundCaseF3-II}&\lesssim&\sum_{\substack{ |n_2|> M} } \frac{1}{ \meanv{n_2}^{2}}
\lesssim \frac{1}{M}\,.\\
\eea

The contribution relative to $\sigma=(2,1,3)$ is 
\begin{equation}
\sum_{\substack{ |n_j| \leq N, n_j \neq 0 \\ n_1 + n_2 + n_3 = 0 \\ \max_j(|n_j|) >M} }  
\frac{ |n_1|^{s} |n_2|^{s}  |n_2+n_3|^{s} |n_1+n_3|^{s}  }{ |n_1|^{2s +1} |n_2|^{2s +1} |n_3|^{2s +1} }\,.
 \label{2BoundCaseF3OLD}
\end{equation}
Using 
$$
|n_2+n_3|^{s} |n_1+n_3|^{s} \lesssim_s (|n_2|^s+|n_3|^{s}) (|n_1|^s+|n_3|^{s}) 
$$
and exchanging the indices, we can reduce \eqref{2BoundCaseF3OLD} to a sum of terms of the form
\bea
\sum_{\substack{ |n_j| \leq N, n_j \neq 0 \\ n_1 + n_2 + n_3 = 0 \\ \max_j(|n_j|) >M} }  
\frac{ |n_1|^{s} |n_2|^{s}  |n_3|^{2s} }{ |n_1|^{2s +1} |n_2|^{2s +1} |n_3|^{2s +1} }&\lesssim& \sum_{\substack{ |n_j| \leq N \\ n_1 + n_2 + n_3 = 0 \\ \max_j(|n_j|) >M} }  
\frac{1}{ \meanv{n_1}\meanv{n_2}^{s +1} \meanv{n_3}^{s +1}}\nn\\
&\leq&\sum_{\substack{\max{(|n_1|,|n_2|)} \gtrsim M} }  
\frac{1}{ \meanv{n_1}\meanv{n_2}^{s +1} \meanv{n_1-n_2}^{s +1}} \nn\\
&\leq&\sum_{|n_1|\gtrsim M}  
\frac{1}{ \meanv{n_1}}\sum_{n_2\in\Z}\frac{1}{\meanv{n_2}^{s +1} \meanv{n_1-n_2}^{s +1}}\label{2BoundCaseF3}\\
&+&\sum_{ n_1\in\Z}  
\frac{1}{ \meanv{n_1}}\sum_{|n_1|\gtrsim M}\frac{1}{\meanv{n_2}^{s +1} \meanv{n_1-n_2}^{s +1}}\label{2BoundCaseF3-II}\,.
\eea
Again \eqref{2BoundCaseF3} and \eqref{2BoundCaseF3-II} can be estimated by using 
(\ref{eq:case-II-conv}) and (\ref{eq:case-II-convM}). We have
\bea
\eqref{2BoundCaseF3}&\leq&\sum_{|n_1|\gtrsim M}  
\frac{1}{ \meanv{n_1}^{s+\frac32}}\lesssim \frac{1}{M^{s+\frac12}}\,,\\
\eqref{2BoundCaseF3-II}&\leq&\sum_{ |n_1|\gtrsim M}  
\frac{1}{ \meanv{n_1}^{s+\frac32}}+\frac{1}{M^{s+\frac12}}\sum_{ n_1\in\Z}\frac{1}{\meanv{n_1}^{s+\frac32}}
\lesssim\frac{1}{M^{s+\frac12}} \,. 
\eea
\end{proof}


\section{Tail estimates}\label{sect:prob}

The goal of this section is to prove the following proposition, that is the key quantitative estimate 
in the study of the quasi-invariance of $\tilde\g_{s}$.

\begin{proposition}\label{prop:Lpdifficile2}
Let $s > 1$. For all $N\in\N\cup\{\infty\}$ it holds
\be\label{eq:Lpdifficile2}
\left\| F_N   \right\|_{L^p(\tilde\g_{s})}\lesssim C(R) p  \,.
\ee
\end{proposition}

We state and prove immediately two useful tail bounds in view of Proposition \ref{prop:energy}.

\begin{lemma}\label{lemma:subexp-inf}
Let $s>1$ and $\kappa >0$. There is $c(R)>0$ such that for all $N \in \N \cup \{ \infty \}$ we have
\be\label{Lemma:Linfty}
\tilde\g_{s}(\|P_N\partial_xu\|_{L^{\infty}}\geq t^{\kappa})\lesssim e^{-c(R)  t^{2s\kappa}}\,. 
\ee
\end{lemma}
\begin{proof}
First we prove the statement for $s > 3/2$. We bound
\be
\|P_N\partial_xu\|_{L^{\infty}}\leq \sum_{j\in\N}2^j\sup_{x\in\T}|\D_jP_Nu|\leq \sum_{j\in\N}\sum_{|n| \simeq 2^j } 2^{j} |(P_N u)(n)|
\ee
and estimate the deviation probability for the r.h.s.

Let $j_t$ the largest  element of $\N$ such that
\begin{equation}\label{obvPreqInfty}
2^{j} < \frac {t^{\kappa}}{R} \quad \mbox{for} \quad j < j_t;
\end{equation}
we set $j_t =0$ if \eqref{obvPreqInfty} is never satisfied.
We split 
$$
\sum_{j\in\N}\sum_{|n| \simeq 2^j } 2^{j} |(P_N u)(n)| \leq 
\sum_{0 \leq j < j_t}\sum_{|n| \simeq 2^j } 2^{j} |(P_N u)(n)|
+
\sum_{j \geq j_t}\sum_{|n| \simeq 2^j } 2^{j} |(P_N u)(n)|\,. 
$$
The first summand is easily evaluated. Indeed
\be\label{eq:holder} 
\sum_{0 \leq j < j_t}\sum_{|n| \simeq 2^j } 2^{j} |(P_N u)(n)| \leq 2^{\frac{3j_t}{2}}\|\D_j u\|_{L^{2}}\lesssim 2^{j_t}\|\D_j u\|_{\dot H^{\frac12}}
\leq2^{j_t}R\leq t^\kappa\,,
\ee
where we first used the Cauchy--Schwartz and then the Bernstein inequality. Since the above inequality 
holds $\tilde \g_{s}$-a.s. we have
\be\label{eq:L-e1}
\tilde \g_{s}\Big(  \sum_{0 \leq j < j_t}\sum_{|n| \simeq 2^j } 2^{j} |(P_N u)(n)| \geq t^{\kappa}  \Big)=0\,.
\ee

To estimate the contribution for $j \geq j_t$ we 
introduce a sequence $\{ \sigma_j \}_{j \geq j_t}$ defined as  
\be\label{eq:sigma}
\s_j:=c_0(j+1-j_t)^{-2}\,, 
\ee 
where $c_0>0$ is sufficiently small in such a way that $\sum_{j \geq j_t}\s_j\leq 1$. 
Then we bound
\be\label{eq:L-interme}
\tilde\g_{s}\Big(\sum_{|n| \simeq 2^j } 2^{j} |(P_N u)(n)|\geq t^{\kappa}\Big)\leq \sum_{j\geq j_t}\tilde\g_{s}\Big(\sum_{|n| \simeq 2^j } |(P_N u)(n)|\geq 2^{-j}\s_jt^{\kappa}\Big)\,.
\ee
As the $\g_{s}$-expectation of $\sum_{|n| \simeq 2^j } |(P_N u)(n)|$ is bounded by 
$C2^{j\left(\frac12 - s\right)}$ and
\begin{equation}\label{Eq:Quatif}
C2^{j\left(\frac12 - s\right)} \leq \frac12 2^{-j}\s_jt^{\kappa} \quad \mbox{for $s > 3/2$ and $t > C_s$,}
\end{equation}
where $C_s$ is a sufficiently large constant (only depending on $s$)\footnote{For instance we can quantify $C_s$ as follows 
$C_s^{\kappa} = \sup_j \frac{2C}{c_0} 2^{j\left( \frac32 - s\right)}(j+1)^2$, noting that this is finite for $s > 3/2$. The restriction 
$t>C_s$ is harmless as the statement \eqref{Lemma:Linfty} is trivial for $t < C_s$.},  
we have as consequence of inequality~\eqref{eq_Hoeffing}
that
\bea\label{eq:5}
\tilde\g_{s}\Big(\sum_{|n| \simeq 2^j } |(P_N u)(n)|\geq 2^{-j}\s_jt^{\kappa}\Big)& \leq &C\exp\left(-c\frac{2^{-2 j}\s_j^2t^{2\k}}{\sum_{|n| \simeq 2^j}n^{-2s-1}}\right)\nn\\
&\leq&C\exp\left( -ct^{2\k}\s_j^22^{2j(s-1))} \right)\,;
\eea
Thus
\be\label{RestrOnS}
\mbox{r.h.s. of (\ref{eq:L-interme})}\lesssim\sum_{j\geq j_t}e^{-c \s_{j}^{2} 2^{j(2s-2)}t^{2 \kappa}}
\lesssim
e^{-c  2^{j_t(2s - 2)}t^{2 \kappa}}
\lesssim e^{-c\frac{t^{ 2s\kappa }}{R^{2s-2}}}\,,
\ee
that concludes the proof when $s >3/2$ (note that the second inequality is in fact justified as long as $s>1$). To handle the case $1 < s \leq 3/2$ we note that the above 
argument still works as long as we restrict to frequencies $j$ such that\footnote{Note that the second inequality is true for 
$j$ sufficiently large. This is again harmless since small frequencies are handled using \eqref{obvPreqInfty}, \eqref{eq:holder}.} 
\begin{equation}
C2^{j\left(\frac12 - s\right)} <  \frac12 2^{-j(1 + \varepsilon)}t^{\kappa} < \frac12 2^{-j}\s_jt^{\kappa};
\end{equation}
where $\varepsilon >0$ will be later chosen sufficiently small; in fact such that $\varepsilon < s - \frac12$ (we could then have chosen the sequence $\s_j$ so that
$2^{-j \varepsilon} < \s_j$).
Thus, denoting with $j^{*}_t$ the largest integer such that 
\begin{equation}\label{sttis}
C2^{ j^{*}_t \left(\frac12 - s\right)} <  \frac12 2^{-j^{*}_t(1 + \varepsilon)} t^{\kappa} 
\end{equation}
 holds, we need to 
handle the frequencies $j > j^{*}_t$ (we set $j^{*}_t =0$ if \eqref{sttis} is never satisfied). 
Namely, it suffices to show 
$$
\tilde\g_{s}(\|  (\Id - P_{2^{j^{*}_t}}) P_N \partial_x u \|_{L^{\infty}}\geq t^{\kappa})
\lesssim e^{-ct^{2s\kappa}}\,.
$$
Note that for $s>1/2$ we have $j^{*}_t \gg_{t} j_t$ (in fact we gain a power of $t$ working with 
$j^{*}_t$ in place of~$j_t$). We have by definition of 
$j^{*}_t$ that
\begin{equation}\label{sttis2}
2^{j^{*}_t \left(\frac12 - s\right)} \gtrsim  2^{-j^{*}_t(1 + \varepsilon)}  t^{\kappa}, 
\end{equation}
namely
\begin{equation}\label{sttis3}
2^{j^{*}_t} \gtrsim  t^{\frac{\kappa}{\frac32 + \varepsilon - s}}. 
\end{equation}
Now we bound as above
\begin{equation}\label{eq:L-interme2}
\tilde\g_{s}(\|  (\Id - P_{2^{j^{*}_t}}) P_N \partial_x u \|_{L^{\infty}} \leq \sum_{j > j^{*}_t}
\tilde\g_{s}(\|  \Delta_{j} P_N \partial_x u \|_{L^{\infty}} > \sigma_j t^{\kappa}).
\end{equation}
Then we use that for all $\varepsilon' >0$ we have 
\begin{equation}\label{AddREf?}
\g_{s}(\|  \Delta_{j} P_N \partial_x u \|_{L^{\infty}} > \sigma_j t^{\kappa}) \lesssim 
C\exp\left( -ct^{2\k} \s_j^2 2^{2j(s-1 - \varepsilon') )} \right).
\end{equation}
This is true since $\| \Delta_{j} P_N \partial_x u \|_{L^2(\g_s)} \simeq C 2^{j(1-s)}$. Using this fact we can show 
that for all $q >2$ we have $\| \Delta_{j} P_N \partial_x u \|_{L^q(\g_s)} \simeq \sqrt{q} 2^{j(1-s)}$. 
Using this 
and the Minkowski's integral inequality we can prove that for all $p < \infty$ we have   
$$\| \| \Delta_{j} P_N \partial_x u \|_{L^p_x} \|_{L^q(\g_s)} \simeq \sqrt{q} 2^{j(1-s)}, \quad \mbox{for all $q > p$.}$$ 
Using this estimate for the momenta, we can prove the following tail bound
$$
\g_{s}(\|  \Delta_{j} P_N \partial_x u \|_{L^{p}} > \sigma_j t^{\kappa}) \lesssim 
C\exp\left( -ct^{2\k} \s_j^2 2^{2j(s-1) )} \right).
$$ 
From this inequality and  
$\|  \Delta_{j} P_N \partial_x u \|_{L^{\infty}} \lesssim 2^{j/p} \|  \Delta_{j} P_N \partial_x u \|_{L^{p}}$ we 
get \eqref{AddREf?}.
Thus
\be
\mbox{r.h.s. of (\ref{eq:L-interme2})}\lesssim\sum_{j > j_t}e^{-c \s_{j}^{2} 2^{j(2s-2- 2 \varepsilon')}t^{2 \kappa}}
\lesssim
e^{-c  2^{j^{*}_t(2s - 2 - 2 \varepsilon')}t^{2 \kappa}} \lesssim
e^{-c  t^{\frac{\kappa}{ 3/2 + \varepsilon - s} (2s - 2 -2 \varepsilon')} t^{2 \kappa}} 
\lesssim e^{-ct^{2s\kappa}},
\ee
where: in the second inequality we used $2s - 2 - 2 \varepsilon' >0$ (which is true for all $s >1$ taking $\varepsilon'$
sufficiently small. In the third inequality we used \eqref{sttis3}. The last inequality holds as long as 
$s > \frac12 + \varepsilon$ and $\varepsilon' \ll \varepsilon$. Since we are allowed to take the $\varepsilon, \varepsilon'$
parameter arbitrarily small, the proof is complete (we recall that we always have a restriction $s > 1$ coming from 
the second inequality in \eqref{RestrOnS}).

\end{proof}

\begin{lemma}\label{lemma:subexpXH}
Let $\kappa >0$ and
\be
a:=\frac{4\kappa s}{2s-1}\,,\qquad b:=\frac{2}{2s-1} 
\ee
Then there are $C,c>0$, such that the bound
\be\label{eq:subexpXH}
\tilde\g_{s}(\|P_Nu\|_{H^s}\geq t^\k)\leq C\exp\left(-c\frac{t^{a}}{R^{b}}\right)\,,
\ee
holds for all $N \in \N $ and $t \gtrsim( \ln N)^{\frac2\kappa}$\,.
%
\end{lemma}
\begin{proof}

Let
$$
X_N^{(s)}:=\sum_{j\in\Z} 2^{js}\|\D_jP_Nu\|_{L^2}\,.
$$
We prove the statement for $X_N^{(s)}$ in place of 
$\|P_Nu\|_{H^s}$. This is sufficient since $\|P_Nu\|_{H^s} \leq X_N^{(s)}$.
Let $j_t$ the largest element of $\N \cup \{ \infty \}$
such that 
\begin{equation}\label{obvPreq}
2^{j (s-\frac12)} < \frac{t^\k}{2 R} \quad \mbox{for} \quad j < j_t \,. 
\end{equation}
We split 
\be\label{eq:X++}
X^{(s)}_N=\sum_{0\leq j < j_t} X^{(s)}_{j,N} 
 +\sum_{j \geq j_t} X^{(s)}_{j,N} \,.
\ee
The first summand is easily evaluated. Indeed
since $s > \frac12$ we have that
$$
\sum_{0\leq j < j_t} X^{(s)}_{j,N}  \leq \sum_{0 \leq j <  j_t} 2^{j(s-\frac12)} \|\D_j u\|_{\dot H^{\frac12}} < \frac{t^{\k}}{R} R = t^{\k}, 
\quad \mbox{(we used \eqref{obvPreq})}\,,
$$
holds $\tilde \g_{s}$-a.s., therefore
\be\label{eq:X-e1}
\tilde\g_{s}\Big(\sum_{0 \leq j < j_t} X^{(s)}_{j,N} \geq t^{\k}\Big)=0\,.
\ee

Let $\sigma_j$ as in \eqref{eq:sigma}. 
We have  
\begin{align}\label{eq:sub-exp-toprove}
\tilde\g_{s}( X^{(s)}_{j,N} \geq \s_j t^\k)
& =  \tilde\g_{s}( 2^{js}\|\D_jP_Nu\|_{L^2} \geq \s_jt^{\k})
\\ \nonumber
& =  \tilde\g_{s}( \|\D_jP_Nu\|^2_{L^2} \geq 2^{-2js} \s_j^2 t^{2\k} )
\,.
\end{align}
We note that
\be
E_s[\|\D_jP_Nu\|^2_{L^2}]=\sum_{\substack{ n\simeq 2^j \\ |n|\leq N}}E_s[|\hat u(n)|^2]\simeq\sum_{\substack{ n\simeq 2^j \\ |n|\leq N}} \frac{1}{n^{2s+1}}\simeq 2^{-2js}1_{\{j\leq\log_2 N\}}\,. 
\ee
Therefore
\bea
\tilde\g_{s}( \|\D_jP_Nu\|^2_{L^2} \geq 2^{-2js} \s_j^2 t^{2\k} )&\lesssim& \tilde\g_{s}( \|\D_jP_Nu\|^2_{L^2} \geq E_s[\|\D_jP_Nu\|^2_{L^2}]+ 2^{-2js} \s_j^2 t^{2\k} )\nn\\
&\lesssim& \tilde\g_{s}( |\|\D_jP_Nu\|^2_{L^2} - E_s[\|\D_jP_Nu\|^2_{L^2}]|\geq 2^{-2js} \s_j^2 t^{2\k} )
\eea
This last term can be bounded by the Bernstein inequality \eqref{eq_Bernstein}. We conclude
\begin{align}
\tilde\g_{s}( X^{(s)}_{j,N} \geq \s_j t^\k)
&\leq C\exp\left(-c\min\left(\s_j^2t^{2\k} 2^{j},\s_j^4t^{4\k}2^{j}\,\right)\right)\,.
\end{align}
Therefore for any fixed $j \geq j_t$ we have
\begin{equation}\label{Using1WC}
\tilde\g_{s}( X^{(s)}_{j,N} \geq \s_jt^{\k})\leq 2e^{-\s_j^2t^{2\k} 2^{j}}
\end{equation}
provided 
$$
t\gtrsim j^{\frac {2}\k} 
\,. 
$$
The estimate extends to all $j\geq j_t$ for 
$$
t\gtrsim \max_{j_t\leq j \lesssim \ln N} 
\gtrsim
(\ln N)^{\frac {2}\k}\,. 
$$ 
Note that by definition of $j_t$ (see \eqref{obvPreq}) we have
\begin{equation}\label{Using2WC}
2^{j_t} \gtrsim \left( \frac{t^\k}{2 R} \right)^{\frac{1}{s-\frac12}}. 
\end{equation}
Thus, using \eqref{Using1WC}-\eqref{Using2WC} we can estimate 
\bea
\tilde\g_{s}\Big(\sum_{j \geq j_t} X^{(s)}_{j,N} \geq t^\k\Big) 
&\leq& \sum_{j \geq j_t} \tilde\g_{s}\Big(X^{s}_{j,N} \geq \s_jt^\k\Big)
\nn\\
&\lesssim& \sum_{j \geq j_t} e^{-\s_j^2t^{2\k} 2^{j}}
\nn\\
&\lesssim&  e^{- t^{2\k} 2^{j_t}}
\lesssim \exp\left(-c\frac{t^{4\k s/(2s-1)}}{R^{2/(2s-1))}}\right)\,\nn
\eea
for some absolute constants $C,c>0$. 
\end{proof}

\begin{remark}
By the same proofs one can show the statements of Lemma \ref{lemma:subexp-inf} and Lemma \ref{lemma:subexpXH} for all the measures $\tilde\g_{s,M}(A)=E_s[1_{\{\mc E[P_M u]\leq R\}\cap A}]$ with $M\geq N$ (we only dealt with $M=\infty$).
\end{remark}

The small deviations of $F_N$ are evaluated as follows. First of all we note that by Proposition \ref{prop:Wick} there is $C>0$ such that for any $M \geq N \in\N$
\begin{equation}\label{Hyp}
\left\|F_N-F_M\right\|_{L^2(\g_s)}\leq \frac{C}{N^{\upsilon}}\,,\quad \upsilon:=\min(\frac12,\frac{2s-1}{4})\,. 
\end{equation}
Since $F_N$ is a trilinear form of Gaussian random variables, we see that \eqref{Hyp} implies by hyper-contractivity that for all $p>2$ there is $C>0$ (possibly different from above) for which
\be\label{eq:hyper}
\left\|F_N-F_M\right\|_{L^p(\g_s)}\leq \frac{Cp^{\frac32}}{N^{\upsilon}}\,.
\ee

\ni From \eqref{eq:hyper} we obtain the following result in the standard way

\begin{proposition}\label{prop:exp0}
Let $s>1$ and $N\in\N$. There are $C,c>0$ such that
\be\label{eq:conc1}
\g_s\left(|F_N-F|\geq t\right)\leq Ce^{-ct^{\frac23} N^{\frac{2\upsilon}{3}}}\,. 
\ee
\end{proposition}

\begin{proposition}\label{prop:gamma1-smalldev}
Let $s>1$ and $t\leq N^{2\upsilon}$. There are $c,C>0$ such that
\be
\tilde\g_{s}(|F_N|\geq t)\leq Ce^{-ct}\,. 
\ee
\end{proposition}
\begin{proof}
Let us set $T:=\lfloor t^{\frac1{2\upsilon}}\rfloor$ and notice that $T\leq N$. By the union bound
\be\label{eq:subexp2-use}
\tilde\g_{s}(|F_N|\geq t)\leq \g_{s}(|F_N-F_T|\geq t/2)+\tilde\g_{s}(|F_T|\geq t/2)\,. 
\ee
By Proposition \ref{prop:exp0} we have
\be
\g_{s}(|F_N-F_T|\geq t)\leq Ce^{-c t^{\frac23} T^{\frac{2\upsilon}{3}}}\leq Ce^{-ct}\,. 
\ee
On the other hand since $t \geq T^{2\upsilon}$ the estimate of Proposition \ref{prop:gamma1-largedev} 
applies to the second summand of \eqref{eq:subexp2-use}. This concludes the proof. 
\end{proof}

\ni Now we complete the proof of Proposition \ref{prop:Lpdifficile2}, combining Proposition \ref{prop:gamma1-largedev} and the following result describing larger deviations of $F_N$.

\begin{proposition}\label{prop:gamma1-largedev}
Let $s>1$ and $t\geq N^{2\upsilon}$. There are $c(R),C>0$ such that
\be
\tilde\g_{s}(|F_N|\geq t)\leq Ce^{-c(R)t}\,. 
\ee
\end{proposition}

\begin{proof}
By Proposition \ref{prop:energy}
$$
\tilde\g_{s}(|F_N|\geq t)\leq \tilde\g_{s} \left( \|P_N u\|_{H^s}^2 \|\partial_xP_N u\|_{L^{\infty}} \geq t \right)\,.
$$
Using Lemma \ref{lemma:subexp-inf} and Lemma \ref{lemma:subexpXH} we get for $\hat\k\in(0,\frac12)$
\begin{align}
\tilde\g_{s} \left( \left( \|P_N u\|_{H^s}^2 \right)^2 \, \|\partial_xP_N u\|_{L^{\infty}} \geq t \right)
&\leq \tilde\g_{s} \left( \|P_N u\|_{H^s}^2 \geq t^{\hat \k} \right) 
+ \tilde\g_{s} \left( \|\partial_xP_N u\|_{L^{\infty}} \geq t^{1-2 \hat \k} \right)\nn\\
&\leq C\exp(-c(R)t^{\hat a})\,,\label{eq:ahat12}
\end{align}
with 
\be
\hat a:=\min\left(\frac{4s}{2s-1}\hat \k\,, \ \frac{4s}{3-1}(1-2 \hat \k)\right)\,.
\ee
Now we optimize 
\begin{equation}\label{SceltaDiKappa2}
\hat \k:= \frac{2s - 1}{4s}, \qquad (\mbox{note $\hat \k \in (0,1/2)$ for $s > 1/2$});
\end{equation}
in particular $\hat a = 1$ and the proof is concluded. 
\end{proof}


\section{Quasi-invariant measures}\label{section:quasi1}

In this section we complete the proof of Theorem \ref{TH:quasi} by the method of \cite{sigma}.

Let us first introduce the set
$$
E_N:=\Span_\R\{(\cos(nx),\sin(nx))\,,\quad |n|\leq N, \quad n\neq0\} \,.
$$
Note $\dim E_N = 2N$.
We denote by $E_N^{\perp}$ the orthogonal complement of $E_N$ in the topology of $L^2(\T)$. 
Letting $\g_{s, N}^{\perp}$ the measure induced on $E_N^{\perp}$ by the map 
\begin{equation}\label{Def:gammaK2}
\varphi_s(\omega,x)=\sum_{|n| > N}\frac{g_n(\omega)}{|n|^{s+1/2}}e^{inx},
\end{equation}
the measure $ \g_{s}$ factorises over $E_N \times E_N^{\perp}$ as
\begin{equation}\label{GammaPerp}
 \g_{s} (du) := \frac{1}{Z_N} e^{-\frac12  \| P_N u \|_{H^{s+ \gamma/2}}^2 } L_N (d P_N u) \,  \g_{s, N}^{\perp}(d P_{>N} u),
\end{equation}
where $L_N$ is the Lebesgue measure induced on $E_N$ by the isomorphism between~$\R^{2N}$ and $E_N$ and 
$Z_N$ is a renormalisation factor. This factorisation is useful since we know by \cite[Lemma 4.2]{sigma} that the Lebesgue measure $L_N$ is invariant under
$\Phi_t^N P_N = P_N \Phi_t^N  $. 

The first important step toward the proof of the quasi-invariance of $\tilde\g_s$ is the following

\begin{proposition}\label{prop:quasi-invN}
Let $N \in \N$, $s>1$ and $R >0$.   
There exists $C(R)$ such that
\be\label{eq:quasi-invN}
\frac{d}{dt} \left( \tilde\g_{s}(\Phi_t^N (A)) \right)^{\frac{1}{p}} \leq C(R) p,
\ee
for all measurable set $A$ and for all $p \geq 1$.
\end{proposition}

\begin{proof}
Using the definition \eqref{Intro:eq:ref-meas}, the factorisation \eqref{GammaPerp} and Proposition 4.1 of \cite{sigma}, we have for all measurable $A$  
\begin{align}\label{STTG}
& \tilde\g_{s} \circ \Phi_t^N (A)  = \int_{\Phi_t^N (A)} \g_{s}(du)  1_{\left\{ \mc E[u] \leq R\right\}} 
\\ \nonumber
& = \int_{A}  L_N (d P_N u)   \g_{s, N}^{\perp} (d P_{>N} u)
1_{\left\{ \mc E[u] \leq R\right\}}
  \exp \left( -\frac{1}{2} \| P_N \Phi_t^N u  \|_{H^{s+\frac12}}^2  \right)  
\\ \nonumber
& =
\int_{A} \tilde \g_{s}(du)  
\exp\left(\frac12 \| P_N u \|^2_{H^{s+\frac12}}-\frac12\| P_N \Phi_t^N u\|^2_{H^{s+\frac12}}\right)
\end{align}
where we used that the Jacobian determinant is unitary
(see \cite[Lemma~4.2]{sigma}) and in the  
second identity we used $\mc E[\Phi_t^N u ] =  \mc E[u ]$; see Lemma 2.4 in \cite{sigma}. 
Since 
$$t \in (\mathbb{R}, +) \to \Phi_t^N$$ 
is a one parameter group of transformations, we can easily check that
\begin{equation}\label{EqualityAtT=0}
\frac{d}{d t} \left( \tilde\g_{s} \circ \Phi_t^N (A') \right)\Big|_{t=\bar t} = 
\frac{d}{d t}\left( \tilde\g_{s} \circ \Phi_t^N (\Phi_{\bar t}^N A' ) \right) \Big|_{t=0} \, 
\end{equation}
for all measurable $A'$.
Using \eqref{EqualityAtT=0} and \eqref{STTG} under the choice $A= \Phi_{\bar t}^N A'$, we arrive to
\begin{align}
\frac{d}{d t} & \left( \tilde\g_{s} \circ \Phi_t^N (A) \right)\Big|_{t=\bar t}
\nn \\ 
&=\frac{d}{dt}\int_{\Phi^N_{\bar t}(A)} \exp\left(\frac12 \| P_N u \|^2_{H^{s+\frac12}}-\frac12\| P_N \Phi_t^N u\|^2_{H^{s+\frac12}}\right) \tilde \g_{s}(du)\Big|_{t=0}
\nn\\ 
&=- \frac12 \int_{\Phi^N_{\bar t}(A)} \tilde\g_{s}(du) \frac{d}{dt}
\| P_N \Phi_t^N u\|^2_{H^{s+\frac12}}\Big|_{t=0}
\nn\\ \label{FirstSummand}
&=- \frac12 \int_{\Phi^N_{\bar t}(A)} \tilde\g_{s}(du) 
F_N\,.
\end{align}

By the H\"older inequality and Proposition \ref{prop:Lpdifficile2}, we get   

\begin{equation}\label{ThirdSup}
\left| \int_{\Phi^N_{\bar t}(A)} \tilde\g_{s}(du) F_N \right|
\leq C(R) \, p \, 
(\tilde\g_{s}(\Phi^N_{\bar t}(A)))^{1-\frac1p}.
\end{equation}
Thus we conclude that there is $C(R)$ such that
\be\label{eq:yud1}
\frac{d}{d t} \left( \tilde\g_{s} \circ \Phi_t^N (A) \right)\leq C(R) \, p \, (\tilde\g_{s}(\Phi^N_{t}(A)))^{1-\frac1p}\,.
\ee
From (\ref{eq:yud1}) we get \eqref{eq:quasi-invN}.
\end{proof}

We are now able to control quantitatively the growth in time of
$\tilde\g_s(\Phi_{t}(A))$.

\begin{proposition}\label{lemma:radiceNLp}
Let $s > 1$ and $R > 0$. There 
exists~$C(R)>1$ such that 
\begin{equation}\label{eq:radiceLpDifficile}
\tilde \g_s(\Phi_{t}(A)) \lesssim  
\tilde \g_s(A)^{ ( C(R)^{-|t|} )}, \qquad \forall t \in \R \,.
\end{equation} 
\end{proposition}

\begin{proof}
We will prove that 
\begin{equation}\label{TPTWU}
\tilde \g_s(\Phi_{t}^N(A)) \lesssim  
\tilde \g_s(A)^{ ( C(R)^{-|t|} )}, \qquad \forall t \in \R \,,
\end{equation}
with   
a constant $C(R)$ which is independent on $N \in \N$. We can then promote this inequality 
to the case $N = \infty$, namely to \eqref{eq:radiceLpDifficile}, proceeding as in Section 8 of \cite{sigma}.

To prove \eqref{TPTWU} we rewrite \eqref{eq:yud1} as
$$
\frac{d}{d t} \left( \left( \tilde\g_{s} \circ \Phi_t^N (A) \right)^{\frac1p} \right) \leq C(R) 
$$
and integrating this over $[0,t]$ we get 
$$
(\tilde\g_{s} \circ \Phi_t^N )(A)\leq (C(R) |t| + \tilde\g_{s} (A)^{\frac1p})^p
$$
It will be useful to rewrite the r.h.s. as  
\be
( |t| +  \tilde \g_s (A)^{\frac1p})^p 
=   \tilde\g_s( A)\left(1+\frac{C(R) |t|}{ \tilde \g_s(A)^{\frac1p}}\right)^p
=  \tilde\g_s(A)e^{p\log\left(1+  C(R) |t|   \tilde \g_s (A)^{-\frac1p}\right)}\,.
\ee
Now we can pick
\be \label{eq:p(A)}
p=p(A) =  \log\frac{1}{\tilde\g_s(A)}\quad\mbox{in such a way that} \quad \tilde\g_s(A)^{-\frac1p} =  e\,.
\ee
Thus
\be\label{Plug1}
(\tilde \g_{s} \circ \Phi_t^N)(A)\leq  \tilde\g_s(A)e^{p \log\left(1+ C(R) e |t| \right)}
\leq  \tilde\g_s(A)e^{p \, C(R) e |t|}\,. 
\ee
Then we claim that for $|t|$ small enough
\be\label{Plug2}
e^{p \, C(R) e |t|}
\leq 
\tilde\g_s(A)^{-1/2}\,.
\ee
To have that, it must be
\be\label{ChoiceOfAlpha}
p \, C(R) e |t| \leq \frac12 \log\frac{1}{\tilde\g_s(A)} = \frac{p}{2}
\ee
which is true for $|t| \leq \frac{1}{2 e C(R)}$.
Plugging \eqref{Plug2} into \eqref{Plug1} we arrive to
\be\label{eq:quasi-quantN}
(\tilde \g_{s} \circ \Phi_t^N)(A) 
\leq  \tilde\g_s(A)^{1/2}, \qquad |t| \leq \frac{1}{2 e C(R)} \,.
\ee
Then, it is easy to see that the desired bound \eqref{eq:radiceLpDifficile} follows by 
iteration of the estimate \eqref{eq:quasi-quantN} (the constants $C(R)$ 
in \eqref{eq:radiceLpDifficile}-\eqref{eq:quasi-quantN} differs by an irrelevant factor). For details we refer to the proof of Lemma 3.3 in \cite{gauge}. 
\end{proof}

By the previous result,
the flow $\Phi_t$ maps 
zero measure sets into zero measure sets, for all $t \in \R$.
Therefore, for $s > 1$, 
we have proved that  $\tilde\g_s\circ\Phi_t$ is absolutely continuous w.r.t. $\tilde\g_s$ 
(and so w.r.t~$\g_{s}$) with a density $f_{s}(t,u) \in L^1(\tilde\g_s)$. 
We finish with a more precise evaluation of the integrability of $f_{s}$. 

\begin{proposition}\label{prop_Lp}
Let $s >1$. There exists $p = p(t, R) > 1$ such that
$f_{s}(t,u) \in L^{p}(\tilde\g_s)$.
\end{proposition}

\begin{proof}
Once we have proved the inequality \eqref{eq:quasi-invN}, 
the proof of the statement proceeds exactly as that of \cite[Proposition 3.4]{gauge} 
(the flow $\mathscr{G}$ must be replaced by $\Phi$). 
From the proof we see that $p=(1-e^{-|t| \ln C(R)})^{-1}$. 
\end{proof}


\end{document}